\documentclass[a4paper, 12pt]{article}

\usepackage{thesisstyle}

\usepackage{authblk}

\begin{document}
   \doublespacing
   
   \title{Noether's Problem on Semidirect Product Groups}
   \author{Huah Chu}   
   \author{
      Shang Huang\footnote{
         Corresponding author. \\
         \emph{E-mail addresses: }
         \href{mailto:hchu@math.ntu.edu.tw}{hchu@math.ntu.edu.tw} (H. Chu), 
         \href{mailto:r02221020@ntu.edu.tw}{r02221020@ntu.edu.tw} (S. Huang)
      }
   }
   \affil{
      \small \emph{
         Department of Mathematics, National Taiwan University, Taipei, Taiwan
      }
   }

   \date{}
   \maketitle
   
\begin{abstract}
   Let $K$ be a field, $G$ a finite group.
   Let $G$ act on the function field $L = K(x_{\sigma} : \sigma \in G)$ by $\tau \cdot x_{\sigma} = x_{\tau\sigma}$ for any $\sigma, \tau \in G$.
   Denote the fixed field of the action by $K(G) = L^{G} = \left\{ \frac{f}{g} \in L : \sigma(\frac{f}{g}) = \frac{f}{g}, \forall \sigma \in G \right\}$.
   Noether's problem asks whether $K(G)$ is rational (purely transcendental) over $K$.
   
   It is known that if $G = C_m \rtimes C_n$ is a semidirect product of cyclic groups $C_m$ and $C_n$ with $\mathbb{Z}[\zeta_n]$ a unique factorization domain, and $K$ contains an $e$th primitive root of unity, where $e$ is the exponent of $G$, then $K(G)$ is rational over $K$.
   
   In this paper, we give another criteria to determine whether $K(C_m \rtimes C_n)$ is rational over $K$.
   In particular, if $p, q$ are prime numbers and there exists $x \in \mathbb{Z}[\zeta_q]$ such that the norm $N_{\mathbb{Q}(\zeta_q)/\mathbb{Q}}(x) = p$, then $\mathbb{C}(C_{p} \rtimes C_{q})$ is rational over $\mathbb{C}$.
   
\end{abstract}

   \clearpage
\section{Introduction}
Let $G$ be a finite group and $K(x_{\sigma} : \sigma \in G)$ be the function field over a field $K$ of $\order(G)$ variables.
The action of $G$ on $K(x_{\sigma} : \sigma \in G)$ is defined as $\tau \cdot x_{\sigma} = x_{\tau\sigma}$ for all $\tau, \sigma \in G$.
The fixed field is
\begin{displaymath}
	K(G) = K(x_{\sigma} : \sigma \in G)^G = \left\{
	   \dfrac{f}{g} \in K(x_{\sigma} : \sigma \in G) : \sigma\left(\dfrac{f}{g}\right) = \dfrac{f}{g}, \forall \sigma \in G
	\right\}.
\end{displaymath}

Noether's problem asks whether $K(G)$ is rational (or, equivalently, purely transcendental) over $K$.

The answer to Noether's problem depends on the group $G$ and also on the field $K$.
In 1915, Fischer \cite{Fischer1915} proved that if $G$ is an abelian group of exponent $e$ and $K$ contains a primitive $e$th root of unity, then $K(G)$ is rational over $K$.

Swan \cite{Swan1983} provided the first counter-example, proving that if $C_n$ is a cyclic group of order $n$, then $\mathbb{Q}(C_n)$ is not rational over $\mathbb{Q}$ when $n = 47, 113$ or $233$, etc.
In fact, when $K$ is any field and $G$ is a finite abelian group, Lenstra \cite{Lenstra1974} has found a necessary and sufficient condition for $K(G)$ to be rational.

The Noether's problem for non-abelian groups is more complicated.

In the case of $p$-groups, Saltman \cite{Saltman1984-npoaacf} gave the first counter-example, proving that there exists a group of order $p^{9}$ which is not rational.
Bogomolov \cite{Bogomolov1988} furthur proved that there exists a group of order $p^{6}$ which is not rational.

However, when $G$ is a $p$-group of order $p^{n}$ and $n \leq 4$, Chu and Kang \cite{Chu2001} has proved that $K(G)$ is rational over $K$ if $K$ contains enough primitive roots of unity.
And Chu, Hoshi, Hu, Kang, Kunyavskii, Prokhorov \cite{Chu2008, Chu2009, Chu2015} discussed the rationality problem when $n = 5$ or $6$.

For direct product of groups, there is a reduction theorem. Let $K$ be any field, $H$ and $G$ be finite groups.
If $K(H)$ is rational over $K$, so is $K(H \times G)$ over $K(G)$.
In particular, if both $K(H)$ and $K(G)$ are rational over $K$, so is $K(H \times G)$ over $K$. \cite{Saltman1982, Kang2009-rtfn}

When $G$ is neither an abelian group nor a $p$-group, the simplest case is $G = C_m \rtimes C_n$, which is a non-abelian semidirect product of cyclic groups $C_m$ and $C_n$.
In 2009, Kang generalized the results of \cite{Haeuslein1971, Hajja1983-riomgola}, proving that
\begin{thm} [\cite{Kang2009-rpfsmg}]
   Let $K$ be a field and $G$ be a finite group. Assume that $G$ contains an abelian normal subgroup $H$ such that $G/H$ is cyclic of prime order $n$. 
   Suppose that a primitive $e$th root of unity $\zeta_{e}$ lies in $K$, where $e$ is the exponent of $G$, and $\mathbb{Z}[\zeta_n]$ is a unique factorization domain. 
   If $G \rightarrow GL(V)$ is a finite-dimensional linear representation of $G$ over $K$, then $K(V)^{G}$ is rational over $K$.
\end{thm}
In particular, $\mathbb{C}(C_m \rtimes C_n)$ is rational over $\mathbb{C}$, if $\mathbb{Z}[\zeta_n]$ is a unique factorization domain.

Note that those integers $n$ for which $\mathbb{Z}[\zeta_n]$ is a unique factorization domain are determined by Masley and Montgomery \cite{Masley1976}. 
In fact, $\mathbb{Z}[\zeta_n]$ is a unique factorization domain if and only if $1 \leq n \leq 22$, or $n = 24$, $25$, $26$, $27$, $28$, $30$, $32$, $33$, $34$, $35$, $36$, $38$, $40$, $42$, $45$, $48$, $50$, $54$, $60$, $66$, $70$, $84$, $90$.

But so far as we know, it is still an open question whether there exists some prime pair $(p, q)$ such that $\mathbb{C}(C_p \rtimes C_q)$ is not rational over $\mathbb{C}$.

The primary concern of this research is to prove that given a pair of primes $(p, q)$, the fixed field $\mathbb{C}(C_p \rtimes C_q)$ could be rational over $\mathbb{C}$ if there exists $x \in \mathbb{Z}[\zeta_{q}]$ with norm $p$ (see Theorem \ref{thm:most_important_case}.) 
However the techniques we used throughout the proof suggested that we may not require $p$ to be a prime.

In this paper, we shall prove the more general case about the rationality of $K(C_m \rtimes C_n)$, as stated in the following Main Theorem.


\begin{main_thm}
	Let $m, n$ be positive integers, where $n$ is an odd prime, $K$ be a field such that neither $m$ nor $n$ is multiple of the characteristic of $K$ and both the primitive roots of unity $\zeta_m, \zeta_n$ lie in $K$.
	Let
	\begin{displaymath}
		G = C_m \rtimes_r C_n = \langle \sigma_1, \sigma_2 : \sigma_1^m = \sigma_2^n = 1, \sigma_2^{-1}\sigma_1\sigma_2 = \sigma_1^r \rangle,
	\end{displaymath}
	where $r^{n} \equiv 1 \pmod{m}$. 
	Let $m' = \frac{m}{\gcd (m, r-1)}$. 
	Suppose there exist relatively prime integers $a_1, \alpha_{n-2}, \ldots, \alpha_0 \in \mathbb{Z}$ such that
	\begin{displaymath}
		a_1m' = \alpha_{n-2} r^{n-2} + \alpha_{n-3} r^{n-3} + \cdots + \alpha_1 r + \alpha_0
	\end{displaymath}
	and
	\begin{displaymath}
		x = \alpha_{n-2}\zeta_{n}^{n-2} + \alpha_{n-3}\zeta_{n}^{n-3} + \cdots + \alpha_1\zeta_{n} + \alpha_0
		   \in \mathbb{Z}[\zeta_{n}]
	\end{displaymath}
	satisfies the norm $N_{\mathbb{Q}(\zeta_{n})/\mathbb{Q}}(x) = m'$.
	Then $K(G)$ is rational over $K$.

	In particular, if $p, q$ are odd primes and there exists $x \in \mathbb{Z}[\zeta_{q}]$ such that $N_{\mathbb{Q}(\zeta_q)/\mathbb{Q}}(x) = p$.
	Then $K(C_p \rtimes C_q)$ is rational over $K$.
	
	Moreover, if $q < 23$, $K(C_{p} \rtimes C_{q})$ is rational over $K$.
\end{main_thm}

The conditions seem to be artificial. But, in fact, we were able to find a class of groups that meet the conditions (see Corollary \ref{cor:class_of_semidirect_prod_of_nonsimple_grps}.)

The problem can be reduced to a Noether-Saltman problem as follows.
Given a finite group $G$ and a $\mathbb{Z}G$-lattice $A$, let $K(A)$ be the quotient field of the group algebra of the free multiplicative abelian group $A$. A Noether-Saltman problem is the following: Is $K(A)^{G}$ rational or stably rational over $K$? (C.f. \cite{Beneish2016})

In the rest of the paper, we first list some preliminary results and lemmas in Section \ref{sec:preliminaries}.
In Section \ref{ch:noethers_problem}, we prove the Main Theorem in the following steps.
Let $G = C_m \rtimes C_n$.
To prove the rationality of $K(G)$, we actually prove the rationality of $K(V)^{G}$ where $V$ is a finite dimensional faithful representation of $G$, such that $K(V)$ is $G$-stably isomorphic to $K(x_{\sigma}:\sigma \in G)$ (see the first paragraph in Section 3.1.) 
Then we reduce the problem to Noether-Saltman problem by showing that $K(V)^{C_m} = K(M)$, where $M$ is a $\mathbb{Z}C_{n}$-lattice and $K(M)$ is the quotient field of the group (Section 3.1.)
We complete the proof by manipulating matrices to find a lattice isomorphic to $M$, whose corresponding fixed subfield is rational over $K$ (Section 3.2, 3.3.)

Finally, we show some consequences of the Main Theorem in Section \ref{sec:cors_and_examples}, and conclude the paper by listing some examples which conform to the conditions of the Main Theorem.


\begin{notations}
Given any $n \in \mathbb{N}$, we use the symbol $[n]$ to denote the set $\{ 1, 2, \ldots, n \}$.
When we write $\{ i_1, \ldots, i_k \}^o \subseteq [n]$,
we shall mean that it's a subset of $[n]$ consisting of elements $i_1, \ldots, i_k$ with order $i_1 < \ldots < i_k$.
If $S$ is a finite set, we shall denote $|S|$ the number of elements in $S$.

Let $\mathcal{A} \in M_{m \times n}(R)$, an $m \times n$ matrix over a commutative ring $R$,
we use the corresponding symbols $\mathfrak{a}_{i}$ and $\mathfrak{a}^{i}$ to denote
the $i$th column vector and row vector of $\mathcal{A}$, respectively.
If $S \subseteq [m]$ and $T \subseteq [n]$, the symbol $\mathcal{A}_{S,T}$ denotes the minor obtained from $\mathcal{A}$ by taking rows in $S$ and columns in $T$.
We also use symbols $\mathcal{A}_{(T)}, \mathcal{A}^{(S)}$ to denote submatrices of $\mathcal{A}$ obtained by deleting columns in $T$ and rows in $S$, respectively.
If $\mathcal{A}$ is a square matrix and $1 \leq i, j \leq n$,
the symbol $\mathcal{A}_{ij}$ denotes the minor of $\mathcal{A}$ obtained by deleting the $i$th row and $j$th column of $\mathcal{A}$.
We shall use the symbol $\det \mathcal{A}$ or $|\mathcal{A}|$ to denote the determinant of $\mathcal{A}$
and the symbol $\adj \mathcal{A}$ to indicate the adjoint matrix of $\mathcal{A}$.

\end{notations}

   \section{Preliminaries} \label{sec:preliminaries}
We recall some preliminary results which will be used in our proof.

\begin{thm}[{\cite{Fischer1915}}] \label{thm:abel_case}
   If $G$ is an abelian group of exponent $e$ $(= \lcm \{ \order(\sigma) : \sigma \in G \})$.
   Suppose a primitive $e$-th root of unity $\zeta_e$ lies in $K$, then $K(G)$ is rational over $K$.
\end{thm}

\begin{thm}[{\cite[Theorem~1]{Hajja1995}}]
   Let $G$ be a finite group acting on $L(x_1, \ldots, x_n)$, the rational function field of $n$ variables over a field $L$.
   Suppose that
   \begin{enumerate}[label=\textnormal{(\alph*)}, nolistsep]
      \item {
         for any $\sigma \in G$, $\sigma(L) \subset L$\textnormal{;}
      }
      \item {
         the restriction of the action of $G$ to $L$ is faithful\textnormal{;}
      }
      \item {
         for any $\sigma \in G$,
         $
            \begin{pmatrix}
               \sigma(x_1) \\ \sigma(x_2) \\ \vdots \\ \sigma(x_n)
            \end{pmatrix}
            = {
               \mathcal{A}(\sigma) \cdot
               \begin{pmatrix}
                  x_1 \\ x_2 \\ \vdots \\ x_n
               \end{pmatrix}
               + \mathcal{B}(\sigma)
            },
         $
      }
   \end{enumerate}
   where $\mathcal{A}(\sigma) \in GL_n(L)$ and $\mathcal{B}(\sigma)$ is an $n \times 1$ matrix over $L$.
   Then there exist elements $z_1, \ldots, z_n \in L(x_1, \ldots, x_n)$, which are algebraically independent over $L$,
   such that $L(x_1, \ldots, x_n) = L(z_1, \ldots, z_n)$ and $\sigma(z_i) = z_i$ for any $\sigma \in G$, any $1 \leq i \leq n$.
\end{thm}

\begin{cor} \label{cor:redu_cor}
   Let $G \rightarrow GL(V)$ be a faithful representation which is irreducible or is a direct sum of inequivalent irreducible representations.
   It induces an action of $G$ on $K(V)$. If the fixed subfield $K(V)^G$ is rational over $K$, then $K(G)$ is rational over $K$.
\end{cor}
\begin{proof}
   It is well known that $V$ can be embedded into the regular representation space $W = \oplus_{\sigma \in G} K \cdot x_{\sigma}$.
   Taking $L = K(V)$ and $K(W) = L(x_1, \cdots, x_n)$ where $n = \dim W - \dim V$.
   Then $K(G) = K(W)^{G}$. By the above theorem, we can find that $K(W) = K(V)(z_1, \cdots, z_n)$ for some $z_i$.
   So $K(G) = K(V)^{G}(z_1, \cdots, z_n)$ and $K(V)^{G}$ is rational imply that $K(G)$ is rational over $K$.
\end{proof}
\begin{thm}[{\cite[Theorem~3.1]{Ahmad2000}}] \label{thm:AHK}
   Let $L$ be any field, $L(x)$ the rational function field of one variable over $L$, and $G$ a finite group acting on $L(x)$.
   Suppose that, for any $\sigma \in G$, $\sigma(L) \subset L$ and $\sigma(x) = a_{\sigma} \cdot x + b_{\sigma}$
   where $a_{\sigma}, b_{\sigma} \in L$ and $a_{\sigma} \neq 0$. Then $L(x)^G = L^G(f)$ for some polynomial $f \in L[x]$.
   In fact, if $m = \min \{ \deg g(x) : g(x) \in L[x]^G \}$,
   then any polynomial $f \in L[x]^G$ with $\deg f = m$ satisfies the property $L(x)^G = L^G(f)$.
\end{thm}

Monomial actions are crucial in solving rationality problem for linear group actions.
A $K$-automorphism $\sigma$ is said to be a monomial automorphism if
\begin{displaymath}
   \sigma \cdot x_j = c_j(\sigma)\prod_{i} x_i^{a_{ij}},
\end{displaymath}
where $(a_{ij})_{1 \leq i,j \leq n} \in GL_n(\mathbb{Z})$ and $c_j(\sigma) \in K \setminus \{ 0 \}$.
If $c_j(\sigma) = 1$ for all $j$ and for all $\sigma \in G$, the action is said to be purely monomial.

It is known that if $L$ is a rational function field of two or three variables over $K$ 
and $G$ is a finite group acting on $L$ by monomial $K$-automorphisms, then the fixed field $L^{G}$ is rational over $K$ \cite{Hajja1983, Hajja1987, Hajja1992, Hajja1994, Hoshi2008}.


Now we list some preliminary lemmas, the proofs are straightforward and omitted.

We first introduce Laplace Expansion. 
Let $S_1, \ldots, S_r$ and $T_1, \ldots, T_r$ be ordered partitions of $[n]$
such that $S_j, T_j$ contain the same number of elements for each $j = 1, \ldots, r$. Write
\begin{displaymath}
   S_j = \{ \alpha_1^{(j)}, \alpha_2^{(j)}, \ldots, \alpha_{k_j}^{(j)} \}^{o}
   \quad \text{ and } \quad
   T_j = \{ \beta_1^{(j)}, \beta_2^{(j)}, \ldots, \beta_{k_j}^{(j)} \}^{o},
\end{displaymath}
then the permutation
$
   \begin{pmatrix}
      S_1 & S_2 & \cdots & S_r \\
      T_1 & T_2 & \cdots & T_r
   \end{pmatrix}
$
is defined to be
\begin{displaymath}
   \begin{pmatrix}
      S_1 & S_2 & \cdots & S_r \\
      T_1 & T_2 & \cdots & T_r
   \end{pmatrix} = {
      \begin{pmatrix}
         \alpha_1^{(1)} & \cdots & \alpha_{k_1}^{(1)} & \cdots & \alpha_{1}^{(r)} & \cdots & \alpha_{k_r}^{(r)} \\
         \beta_1^{(1)} & \cdots & \beta_{k_1}^{(1)} & \cdots & \beta_{1}^{(r)} & \cdots & \beta_{k_r}^{(r)}
      \end{pmatrix}
   } \in \sym_n,
\end{displaymath}
where $\sym_n$ is the symmetric group of order $n$.

\begin{thm}[{Laplace Expansion, \cite[pp.~416-417]{Jacobson2009}}] \label{thm:complete_laplace}
   Let $\mathcal{A}$ be an $n \times n$ matrix over a commutative ring $R$.
   Suppose $\{ S_1, \ldots, S_r \}$ is an ordered partition of $[n]$, then
   \begin{displaymath}
      \det \mathcal{A} = {
         \sum_{\{ T_1, \ldots, T_r \}} {
            \sgn {
               \begin{pmatrix}
                  S_1 & S_2 & \cdots & S_r \\
                  T_1 & T_2 & \cdots & T_r
               \end{pmatrix}
            }
            \mathcal{A}_{S_1,T_1} \cdots \mathcal{A}_{S_r,T_r}
         }
      },
   \end{displaymath}
   where $\{ T_1, \ldots, T_r \}$ runs through all possible ordered partitions of $[n]$ with $|T_j| = |S_j|$ for all $j$.
\end{thm}

\begin{lemma} \label{lemma:det_of_mul_of_irregular_mtx}
   Let $R$ be a commutative ring.
   Suppose $m \leq n$, $\mathcal{A} \in M_{m \times n}(R)$ and $\mathcal{B} \in M_{n \times m}(R)$.
   Let $T = [m]$, then
   \begin{displaymath}
      \det \mathcal{AB} = {
         \sum_{
            \substack {
               S \subseteq [n] \\
               |S| = m
            }
         } \mathcal{A}_{T,S} \mathcal{B}_{S,T}
      }.
   \end{displaymath}
\end{lemma}

\begin{lemma} \label{lemma:mul_of_submtx_del_1st_row_col}
   Let $R$ be a commutative ring, $\mathcal{P} \in M_n(R)$ and $\mathcal{Q} = \adj \mathcal{P}$. Write
   \begin{displaymath}
      \mathcal{P} = [ \mathfrak{p}_{1}, \mathfrak{p}_{2}, \ldots, \mathfrak{p}_{n} ], \quad
      \mathcal{Q} = {
         \left[ \begin{smallmatrix}
            \mathfrak{q}^{1} \\ \vdots \\ \mathfrak{q}^{n}
         \end{smallmatrix} \right]
      }.
   \end{displaymath}
   Let $\mathcal{P}_{(1)}$ be the submatrix of $\mathcal{P}$ obtained by deleting the first column and
   $\mathcal{Q}^{(1)}$ be the submatrix of $\mathcal{Q}$ obtained by deleting the first row. Then
   \begin{displaymath}
      \mathcal{P}_{(1)}\mathcal{Q}^{(1)} = (\det P)I_{n} - \mathfrak{p}_{1}\mathfrak{q}^{1}.
   \end{displaymath}
\end{lemma}

\begin{lemma} \label{lemma:minor_after_row_op}
   Let $R$ be a commutative ring. Let
   \begin{displaymath}
      \mathcal{A} = {
         \left[
         \begin{smallmatrix}
            -a_n & & & & a_1 \\
            & -a_n & & & a_2 \\
            & & \ddots & & \vdots \\
            & & & -a_n & a_{n-1}
         \end{smallmatrix}
         \right]
      } \in M_{(n-1) \times n}(R), 
      \quad
      \mathcal{B} = {
         \left[
         \begin{smallmatrix}
            b_{11} & b_{12} & \ldots & b_{1, n-1} \\
            b_{21} & b_{22} & \ldots & b_{2, n-1} \\
            \vdots & \vdots & & \vdots \\
            b_{n1} & b_{n2} & \ldots & b_{n, n-1} \\
         \end{smallmatrix}
         \right]
      } \in M_{n \times (n-1)}(R)
   \end{displaymath}
   and
   \begin{displaymath}
      \mathcal{C} = {
         \left[
         \begin{smallmatrix}
            a_1 & b_{11} & \ldots & b_{1, n-1} \\
            a_2 & b_{21} & \ldots & b_{2, n-1} \\
            \vdots & \vdots & & \vdots \\
            a_n & b_{n1} & \ldots & b_{n, n-1} \\
         \end{smallmatrix}
         \right]
      } \in M_{n \times n}(R).
   \end{displaymath}
   Then $\det \mathcal{A}\mathcal{B} = a_n^{n-2}\det \mathcal{C}$.
\end{lemma}

\begin{prop} \label{prop:det_compound_matrix}
   Let $\mathcal{A}=(a_{ij}) \in M_n(R)$ and $\mathcal{A}_{ij}$ denote the minor obtained by deleting $i$th row and $j$th column of $\mathcal{A}$.
   Let $\{ i_1, i_2, \ldots, i_{n-2} \}^o$ and $\{ j_1, j_2, \ldots, j_{n-2} \}^o$ be subsets of $[n]$ such that
   \begin{align*}
      &\{ 1, \ldots, n \} \backslash \{ i_1, \ldots, i_{n-2} \} = \{ k_1, k_2 \}^o, \\
      &\{ 1, \ldots, n \} \backslash \{ j_1, \ldots, j_{n-2} \} = \{ l_1, l_2 \}^o.
   \end{align*}
   Then
   \begin{displaymath}
      {
         \left|
         \begin{smallmatrix}
            \mathcal{A}_{i_1j_1} & \mathcal{A}_{i_1j_2} & \cdots & \mathcal{A}_{i_1j_{n-2}} \\
            \mathcal{A}_{i_2j_1} & \mathcal{A}_{i_2j_2} & \cdots & \mathcal{A}_{i_2j_{n-2}} \\
            \vdots & \vdots & & \vdots \\
            \mathcal{A}_{i_{n-2}j_1} & \mathcal{A}_{i_{n-2}j_2} & \cdots & \mathcal{A}_{i_{n-2}j_{n-2}} \\
         \end{smallmatrix}
         \right|
      } =
      {
         (\det \mathcal{A})^{n-3}
         \left|
         \begin{smallmatrix}
            a_{k_1l_1} & a_{k_1l_2} \\
            a_{k_2l_1} & a_{k_2l_2} \\
         \end{smallmatrix}
         \right|.
      }
   \end{displaymath}
\end{prop}
\begin{proof}
\setcounter{step}{0}
\begin{step}
   Consider first that $R = \mathbb{Z}[x_{ij}]$, where $x_{ij}$ are variables for $1 \leq i, j \leq n$.
   Put $\tilde{\mathcal{A}} = (x_{ij})$ and
   \begin{displaymath}
      \tilde{\mathcal{U}} = {
         \left[
         \begin{smallmatrix}
            \tilde{\mathcal{A}}_{i_1,1} & \cdots & \tilde{\mathcal{A}}_{i_1,l_1} & \cdots & \tilde{\mathcal{A}}_{i_1,l_2} & \cdots & \tilde{\mathcal{A}}_{i_1,n} \\
            \tilde{\mathcal{A}}_{i_2,1} & \cdots & \tilde{\mathcal{A}}_{i_2,l_1} & \cdots & \tilde{\mathcal{A}}_{i_2,l_2} & \cdots & \tilde{\mathcal{A}}_{i_2,n} \\
            \vdots & & \vdots & & \vdots & & \vdots \\
            \tilde{\mathcal{A}}_{i_{n-2},1} & \cdots & \tilde{\mathcal{A}}_{i_{n-2},l_1} & \cdots & \tilde{\mathcal{A}}_{i_{n-2},l_2} & \cdots & \tilde{\mathcal{A}}_{i_{n-2},n} \\
            0 & \cdots & x_{k_1,l_1} & \cdots & x_{k_1,l_2} & \cdots & 0 \\
            0 & \cdots & x_{k_2,l_1} & \cdots & x_{k_2,l_2} & \cdots & 0 \\
         \end{smallmatrix}
         \right]
      }
   \end{displaymath}
   and
   \begin{displaymath}
      \tilde{\mathcal{V}} = {
         \left[
         \begin{smallmatrix}
            (-1)^{i_1+1}x_{i_1,1} & \cdots & (-1)^{i_{n-2}+1}x_{i_{n-2},1} & (-1)^{k_1+1}x_{k_1,1} & (-1)^{k_2+1}x_{k_2,1} \\
            (-1)^{i_1+2}x_{i_1,2} & \cdots & (-1)^{i_{n-2}+2}x_{i_{n-2},2} & (-1)^{k_1+2}x_{k_1,2} & (-1)^{k_2+2}x_{k_2,2} \\
            \vdots & & \vdots & \vdots & \vdots \\
            (-1)^{i_1+n}x_{i_1,n} & \cdots & (-1)^{i_{n-2}+n}x_{i_{n-2},n} & (-1)^{k_1+n}x_{k_1,n} & (-1)^{k_2+n}x_{k_2,n} \\
         \end{smallmatrix}
         \right]
      }.
   \end{displaymath}
   We get
   \begin{equation} \label{eq:matrix_uv}
      \tilde{\mathcal{U}}\tilde{\mathcal{V}} = {
         \left[
            \begin{array}{ccc:c}
               \det\tilde{\mathcal{A}} & & & \\
               & \ddots & & \text{\Large 0} \\
               & & \det\tilde{\mathcal{A}} & \\ \hdashline
               & \text{\Large *} & & {
                  \left[
	                  \begin{smallmatrix}
	                     x_{k_1,l_1} & x_{k_1,l_2} \\
	                     x_{k_2,l_1} & x_{k_2,l_2}
	                  \end{smallmatrix}
                  \right]
                  \cdot
                  \left[
	                  \begin{smallmatrix}
	                     (-1)^{k_1+l_1}x_{k_1,l_1} & (-1)^{k_2+l_1}x_{k_2,l_1} \\
	                     (-1)^{k_1+l_2}x_{k_1,l_2} & (-1)^{k_2+l_2}x_{k_2,l_2}
	                  \end{smallmatrix}
                  \right]
               }
            \end{array}
         \right]
      }.
   \end{equation}
   Consider their determinants, we have
   \begin{displaymath}
      \det\tilde{\mathcal{U}} = (-1)^{(n-l_2)+(n-1-l_1)} {
         \left|
         \begin{smallmatrix}
            x_{k_1,l_1} & x_{k_1,l_2} \\
            x_{k_2,l_1} & x_{k_2,l_2} \\
         \end{smallmatrix}
         \right|
         \cdot
         \left|
         \begin{smallmatrix}
            \tilde{\mathcal{A}}_{i_1j_1} & \tilde{\mathcal{A}}_{i_1j_2} & \cdots & \tilde{\mathcal{A}}_{i_1j_{n-2}} \\
            \tilde{\mathcal{A}}_{i_2j_1} & \tilde{\mathcal{A}}_{i_2j_2} & \cdots & \tilde{\mathcal{A}}_{i_2j_{n-2}} \\
            \vdots & \vdots & & \vdots \\
            \tilde{\mathcal{A}}_{i_{n-2}j_1} & \tilde{\mathcal{A}}_{i_{n-2}j_2} & \cdots & \tilde{\mathcal{A}}_{i_{n-2}j_{n-2}} \\
         \end{smallmatrix}
         \right|
      },
   \end{displaymath}
   \begin{align*}
      \det\tilde{\mathcal{V}} &= (-1)^{(n-k_2)+(n-1-k_1)} {
         \det ((-1)^{i+j}x_{ij})
      } \\
      &= {
         (-1)^{(n-k_2)+(n-1-k_1)}\det\tilde{\mathcal{A}}
      }
   \end{align*}
   and from \eqref{eq:matrix_uv}, we have
   \begin{displaymath}
      \det\tilde{\mathcal{U}}\tilde{\mathcal{V}} = {
         (-1)^{k_1+k_2+l_1+l_2}(\det\tilde{\mathcal{A}})^{n-2} {
            \left|
            \begin{smallmatrix}
               x_{k_1,l_1} & x_{k_1,l_2} \\
               x_{k_2,l_1} & x_{k_2,l_2} \\
            \end{smallmatrix}
            \right|^{2}
         }
      }.
   \end{displaymath}

   Combine results above, we obtain
   \begin{displaymath}
      (\det\tilde{\mathcal{A}})^{n-2} {
         \left|
         \begin{smallmatrix}
            x_{k_1,l_1} & x_{k_1,l_2} \\
            x_{k_2,l_1} & x_{k_2,l_2} \\
         \end{smallmatrix}
         \right|^{2}
      }
      = {
         (\det\tilde{\mathcal{A}}) {
            \left|
            \begin{smallmatrix}
               x_{k_1,l_1} & x_{k_1,l_2} \\
               x_{k_2,l_1} & x_{k_2,l_2} \\
            \end{smallmatrix}
            \right|
            \cdot
            \left|
            \begin{smallmatrix}
               \tilde{\mathcal{A}}_{i_1j_1} & \tilde{\mathcal{A}}_{i_1j_2} & \cdots & \tilde{\mathcal{A}}_{i_1j_{n-2}} \\
               \tilde{\mathcal{A}}_{i_2j_1} & \tilde{\mathcal{A}}_{i_2j_2} & \cdots & \tilde{\mathcal{A}}_{i_2j_{n-2}} \\
               \vdots & \vdots & & \vdots \\
               \tilde{\mathcal{A}}_{i_{n-2}j_1} & \tilde{\mathcal{A}}_{i_{n-2}j_2} & \cdots & \tilde{\mathcal{A}}_{i_{n-2}j_{n-2}} \\
            \end{smallmatrix}
            \right|
         }
      }.
   \end{displaymath}
   If $\det \tilde{\mathcal{A}} \not= 0$ and
   $
      \left|
      \begin{smallmatrix}
         x_{k_1,l_1} & x_{k_1,l_2} \\
         x_{k_2,l_1} & x_{k_2,l_2} \\
      \end{smallmatrix}
      \right|
      \not= 0
   $, then the statement holds for $R = \mathbb{Z}[x_{ij}]$.
\end{step}
\begin{step}
   Define the polynomial $g \in \mathbb{Z}[x_{ij}]$ by
   \begin{displaymath}
      g = {
         {
            (\det \tilde{\mathcal{A}})^{n-3}
            \left|
            \begin{smallmatrix}
               x_{k_1l_1} & x_{k_1l_2} \\
               x_{k_2l_1} & x_{k_2l_2} \\
            \end{smallmatrix}
            \right|
         }
         -
         {
            \left|
            \begin{smallmatrix}
               \tilde{\mathcal{A}}_{i_1j_1} & \tilde{\mathcal{A}}_{i_1j_2} & \cdots & \tilde{\mathcal{A}}_{i_1j_{n-2}} \\
               \tilde{\mathcal{A}}_{i_2j_1} & \tilde{\mathcal{A}}_{i_2j_2} & \cdots & \tilde{\mathcal{A}}_{i_2j_{n-2}} \\
               \vdots & \vdots & & \vdots \\
               \tilde{\mathcal{A}}_{i_{n-2}j_1} & \tilde{\mathcal{A}}_{i_{n-2}j_2} & \cdots & \tilde{\mathcal{A}}_{i_{n-2}j_{n-2}} \\
            \end{smallmatrix}
            \right|
         }
      }.
   \end{displaymath}
   Then by Step 1, $g(x_{ij}) = 0$ if $\det \tilde{\mathcal{A}} \not= 0$ and
   $x_{k_1,l_1}x_{k_2,l_2} - x_{k_1,l_2}x_{k_2,l_1} \not= 0$.
   Hence we have $g(\mathfrak{a}) = 0$ for those $\mathfrak{a} \in \mathbb{C}^{n^{2}}$ which do not belong to the union of varieties $V(\det\tilde{\mathcal{A}}) \cup V(x_{k_1,l_1}x_{k_2,l_2} - x_{k_1,l_2}x_{k_2,l_1})$.

   Note that polynomial functions are continuous and all open sets are dense in Zariski topological space $\mathbb{C}^{n^{2}}$. Hence $g = 0$ on open set
   $
      V(\det\tilde{\mathcal{A}})^{c}
      \cap
      V(x_{k_1,l_1}x_{k_2,l_2} - x_{k_1,l_2}x_{k_2,l_1})^{c}
   $, which implies $g = 0$ on the whole space $\mathbb{C}^{n^{2}}$.

   For general commutative ring $R$, we have a well-defined ring homomorphism
   $\phi : \mathbb{Z}[x_{ij}] \rightarrow R$ such that $x_{ij} \mapsto a_{ij}$.
   Apply $\phi$ to $g$, the statement follows.
\end{step}
\end{proof}


\begin{defn} \label{def:wedge}
   Given an $n \times (n-1)$ matrix $\Omega$, define
   \begin{displaymath}
      \wedge^{n-1} \Omega = {
         \langle
            (-1)^{n-1}\omega_1, (-1)^{n-2}\omega_2, \ldots, (-1)\omega_{n-1}, \omega_n
         \rangle
      }, 
   \end{displaymath}
   where $\omega_i$ is the minor of $\Omega$ by deleting the $i$\textnormal{th} row.
\end{defn}

\begin{prop} \label{prop:minor_of_mtx_conjugation}
   Let $R$ be a commutative ring and $n \geq 3$.
   Suppose $\mathcal{A} = (a_{ij}) \in M_n(R)$ and $\mathcal{P} = (p_{ij}) \in SL_n(R)$.
   Define $\tilde{\mathcal{B}} = (b_{ij}) = \mathcal{P}^{-1}\mathcal{A}\mathcal{P}$ and put
   \begin{align*}
      \mathfrak{p} := \mathfrak{p}_{1} = {
         \left[
         \begin{smallmatrix}
            p_{11} \\ p_{21} \\ \vdots \\ p_{n1}
         \end{smallmatrix}
         \right]
      }, \indent
      \mathfrak{q} = {
         \left[
         \begin{smallmatrix}
            b_{21} \\ \vdots \\ b_{n1}
         \end{smallmatrix}
         \right]
      } \text{ and }
      \mathcal{B} = {
         \left[
         \begin{smallmatrix}
            b_{22} & \cdots & b_{2n} \\
            \vdots & & \vdots \\
            b_{n2} & \cdots & b_{nn}
         \end{smallmatrix}
         \right]
      }.
   \end{align*}
   If 
   \begin{align*}
      &\left\langle {
         u_1, u_2, \ldots, u_{n}
      } \right\rangle
      = {
         \wedge^{n-1}[\mathfrak{p}, \mathcal{A}\mathfrak{p}, \mathcal{A}^{2}\mathfrak{p}, \ldots, \mathcal{A}^{n-2}\mathfrak{p}]
      }, \\
      &\left\langle {
         w_1, w_2, \ldots, w_{n-1}
      } \right\rangle
      = {
         \wedge^{n-2}[\mathfrak{q}, \mathcal{B}\mathfrak{q}, \mathcal{B}^{2}\mathfrak{q}, \ldots, \mathcal{B}^{n-3}\mathfrak{q}]
      }, 
   \end{align*}
   then
   \begin{displaymath}
      {
         \left|
         \begin{smallmatrix}
            0 & u_1 & \cdots & u_{n} \\
            p_{11} & a_{11} & \cdots & a_{1n} \\
            \vdots & \vdots & & \vdots \\
            p_{n1} & a_{n1} & \cdots & a_{nn}
         \end{smallmatrix}
         \right|_{(n+1) \times (n+1)}
      }
      = -{
         \left|
         \begin{smallmatrix}
            0 & w_1 & \cdots & w_{n-1} \\
            b_{21} & b_{22} & \cdots & b_{2n} \\
            \vdots & \vdots & & \vdots \\
            b_{n1} & b_{n2} & \cdots & b_{nn}
         \end{smallmatrix}
         \right|_{n \times n}
      }.
   \end{displaymath}
\end{prop}
\begin{proof}
   \setcounter{step}{0}
   \begin{step}
      We shall prove that
      \begin{displaymath} \label{eq:wedge_equality}
         \langle
            0, {
               \wedge^{n-2}[\mathfrak{q}, \mathcal{B}\mathfrak{q}, \mathcal{B}^{2}\mathfrak{q}, \ldots, \mathcal{B}^{n-3}\mathfrak{q}]
            }
         \rangle
         = {
            \wedge^{n-1}[\mathfrak{p}_{1}, \mathcal{A}\mathfrak{p}_{1}, \mathcal{A}^{2}\mathfrak{p}_{1}, \ldots, \mathcal{A}^{n-2}\mathfrak{p}_{1}]
	         \cdot \mathcal{P}
         }.
      \end{displaymath}
      
      We first show that the first component of 
      $
         \wedge^{n-1}[\mathfrak{p}_{1}, \mathcal{A}\mathfrak{p}_{1}, \mathcal{A}^{2}\mathfrak{p}_{1}, \ldots, \mathcal{A}^{n-2}\mathfrak{p}_{1}]
	         \cdot \mathcal{P}
      $ is zero.
      
      By definition, we have
      \begin{align*}
         {
            \wedge^{n-1} [\mathfrak{p}_{1}, \mathcal{A}\mathfrak{p}_{1}, \ldots, \mathcal{A}^{n-2}\mathfrak{p}_{1}]\cdot\mathfrak{b}
         }
         = {
            \left|[\mathfrak{p}_{1}, \mathcal{A}\mathfrak{p}_{1}, \ldots, \mathcal{A}^{n-2}\mathfrak{p}_{1}, \mathfrak{b}]\right|
         }
      \end{align*}
      for all $\mathfrak{b} = \langle b_1, \ldots, b_n \rangle^{t}$.
      Hence, the $i$th component of $\wedge^{n-1} [\mathfrak{p}_{1}, \mathcal{A}\mathfrak{p}_{1}, \ldots, \mathcal{A}^{n-2}\mathfrak{p}_{1}]\cdot \mathcal{P}$ is the determinant
      \begin{equation} \label{eq:(i+1)th_comp_of_wedge_unordered}
         \left|[\mathfrak{p}_{1}, \mathcal{A}\mathfrak{p}_{1}, \ldots, \mathcal{A}^{n-2}\mathfrak{p}_{1}, \mathfrak{p}_{i}]\right|, 
      \end{equation}
      which implies that the first component is zero.
      
      Therefore, it remains to show that the $i$th component of 
      $\wedge^{n-2}[\mathfrak{q}, \mathcal{B}\mathfrak{q}, \mathcal{B}^{2}\mathfrak{q}, \ldots, \mathcal{B}^{n-3}\mathfrak{q}]$
      is equal to the $(i+1)$th component of 
      $
         \wedge^{n-1}[\mathfrak{p}_{1}, \mathcal{A}\mathfrak{p}_{1}, \mathcal{A}^{2}\mathfrak{p}_{1}, \ldots, \mathcal{A}^{n-2}\mathfrak{p}_{1}]
         \cdot \mathcal{P}
	   $.
      
      By \  doing \  proper \  column \  operations \  on \  \eqref{eq:(i+1)th_comp_of_wedge_unordered}, \ 
the \  $(i+1)$th \ component \  of \\
$\wedge^{n-1} [\mathfrak{p}_{1}, \mathcal{A}\mathfrak{p}_{1}, \ldots, \mathcal{A}^{n-2}\mathfrak{p}_{1}]\cdot \mathcal{P}$ becomes
      \begin{equation} \label{eq:(i+1)th_comp_of_wedge_ordered}
         (-1)^{n-i+1}
         \left|
            [\mathfrak{p}_{1}, \mathcal{A}\mathfrak{p}_{1}, \ldots, \mathcal{A}^{i-1}\mathfrak{p}_{1}, \mathfrak{p}_{i+1}, \mathcal{A}^{i}\mathfrak{p}_{1}, \ldots, \mathcal{A}^{n-2}\mathfrak{p}_{1}]
         \right|.
      \end{equation}
      
      Let
      $
         \mathcal{V} = {
            [ 
               \mathcal{A}\mathfrak{p}_{1}, 
               \mathcal{A}^{2}\mathfrak{p}_{1}, 
               \ldots, 
               \mathcal{A}^{n-2}\mathfrak{p}_{1}
            ]
         } \in M_{n \times (n-2)}(R)
      $.
      Apply Laplace Expansion (Theorem \ref{thm:complete_laplace}) 
      to \eqref{eq:(i+1)th_comp_of_wedge_ordered} by taking
      $T' = \{ 1, i+1 \}$ and $T = [n] \backslash T'$, 
      then the $(i+1)$th component of $\wedge^{n-1} [\mathfrak{p}_{1}, \mathcal{A}\mathfrak{p}_{1}, \ldots, \mathcal{A}^{n-2}\mathfrak{p}_{1}]\cdot \mathcal{P}$ becomes
      \begin{align} \label{eq:(i+1)th_comp_of_wedge_in_rhs}
         &\quad (-1)^{n-i+1}\left|
            [\mathfrak{p}_{1}, \mathcal{A}\mathfrak{p}_{1}, \ldots, \mathcal{A}^{i-1}\mathfrak{p}_{1}, \mathfrak{p}_{i+1}, \mathcal{A}^{i}\mathfrak{p}_{1}, \ldots, \mathcal{A}^{n-2}\mathfrak{p}_{1}]
         \right| \nonumber \\
         &= {
            (-1)^{n-i+1}
            \sum_{
               \substack {
                  S = [n] \backslash \{ k, l \}^{o} \\
                  S' = \{ k, l \}^{o}
               }
            } {
               \sgn {
                  \begin{pmatrix}
                     S & S' \\
                     T & T'
                  \end{pmatrix}
               }
               \begin{vmatrix}
                  p_{k1} & p_{k,i+1} \\
                  p_{l1} & p_{l,i+1}
               \end{vmatrix}
               \mathcal{V}_{S,\bar{T}}
            }
         } \nonumber \\
         &= {
            (-1)^{n-i+1}
            \sum_{
               \substack {
                  S = [n] \backslash \{ k, l \}^{o} \\
                  S' = \{ k, l \}^{o}
               }
            } {
               (-1)^{k+l+i}
               \begin{vmatrix}
                  p_{k1} & p_{k,i+1} \\
                  p_{l1} & p_{l,i+1}
               \end{vmatrix}
               \mathcal{V}_{S,\bar{T}}
            }
         }, 
      \end{align}
      where $\bar{T} = [n - 2]$.
      
      On the other hand, consider the $i$th component of 
      $\wedge^{n-2}[\mathfrak{q}, \mathcal{B}\mathfrak{q}, \mathcal{B}^{2}\mathfrak{q}, \ldots, \mathcal{B}^{n-3}\mathfrak{q}]$.
      Write $\mathcal{Q} = \mathcal{P}^{-1} = \adj \mathcal{P}$.
      Let $\mathcal{Q}^{(1)}$ denote the submatrix of $\mathcal{Q}$ by deleting the first row and
      $\mathcal{P}_{(1)}$ denote the submatrix of $\mathcal{P}$ by deleting the first column, we have
      \begin{displaymath}
         \mathcal{B}^{i}\mathfrak{q}
         = (\mathcal{Q}^{(1)}\mathcal{A}\mathcal{P}_{(1)})^{i} (\mathcal{Q}^{(1)}\mathcal{A}\mathfrak{p}_1)
         = \mathcal{Q}^{(1)} (\mathcal{A}\mathcal{P}_{(1)}\mathcal{Q}^{(1)})^{i} \mathcal{A}\mathfrak{p}_1.
      \end{displaymath}
      Let $\mathcal{Q}^{(1,i+1)}$ denote the submatrix of $\mathcal{Q}$ by deleting the first and $(i+1)$th rows,
      then the $i$th component of
      $\wedge^{n-2} [ \mathfrak{q}, \mathcal{B}\mathfrak{q}, \ldots, \mathcal{B}^{n-3}\mathfrak{q} ]$
      is
      \begin{equation} \label{eq:ith_comp_of_wedge}
         (-1)^{n-i+1}
         \left|[
            \mathcal{Q}^{(1,i+1)} \mathcal{A}\mathfrak{p}_1,
            \mathcal{Q}^{(1,i+1)} (\mathcal{A}\mathcal{P}_{(1)}\mathcal{Q}^{(1)}) \mathcal{A}\mathfrak{p}_1,
            \ldots,
            \mathcal{Q}^{(1,i+1)} (\mathcal{A}\mathcal{P}_{(1)}\mathcal{Q}^{(1)})^{n-3} \mathcal{A}\mathfrak{p}_1
         ]\right|.
      \end{equation}
      The equation above can actually be simplified to
      \begin{equation} \label{eq:ith_comp_of_wedge_simplified}
         (-1)^{n-i+1}
         \left|[
            \mathcal{Q}^{(1,i+1)} \mathcal{A}\mathfrak{p}_1,
            \mathcal{Q}^{(1,i+1)} \mathcal{A}^{2}\mathfrak{p}_1,
            \ldots,
            \mathcal{Q}^{(1,i+1)} \mathcal{A}^{n-2}\mathfrak{p}_1
         ]\right|.
      \end{equation}
      Indeed, by Lemma \ref{lemma:mul_of_submtx_del_1st_row_col}, we have
      \begin{equation} \label{eq:ith_comp_of_wedge_original}
         \mathcal{Q}^{(1,i+1)} (\mathcal{A}\mathcal{P}_{(1)}\mathcal{Q}^{(1)})^{j} \mathcal{A}\mathfrak{p}_1
         = {
            \mathcal{Q}^{(1,i+1)} \{ \mathcal{A}(\det \mathcal{P} \cdot I_n - \mathfrak{p}_1 \mathfrak{q}^{1}) \}^{j} \mathcal{A}\mathfrak{p}_1
         }.
      \end{equation}
      Expand it completely, then every monomial is of the following two types
      \begin{enumerate}[label=\textnormal{(\alph*)}, nolistsep]
         \item {
            $\mathcal{Q}^{(1,i+1)} \mathcal{A}^{j+1} \mathfrak{p}_1$, or
         }
         \item {
            $\mathcal{Q}^{(1,i+1)} \mathcal{A}^{k}(\mathcal{A} \mathfrak{p}_1 \mathfrak{q}^{1}) \cdots \mathcal{A} \mathfrak{p}_1$
            for some $0 \leq k < j$.
         }
      \end{enumerate}
      Since $\mathfrak{q}^{1} \cdots \mathcal{A} \mathfrak{p}_1$ is a constant,
      the later is actually of the form
      $c_k \mathcal{Q}^{(1,i+1)} \mathcal{A}^{k+1} \mathfrak{p}_1$ for some $c_k \in R$.
      Thus \eqref{eq:ith_comp_of_wedge_original} is equal to
      \begin{equation} \label{eq:ith_comp_of_wedge_expand}
         \sum_{k=0}^{j-1} c_k \mathcal{Q}^{(1,i+1)} \mathcal{A}^{k+1} \mathfrak{p}_1 +
         \mathcal{Q}^{(1,i+1)} \mathcal{A}^{j+1} \mathfrak{p}_1.
      \end{equation}
      Substitute \eqref{eq:ith_comp_of_wedge_expand} into \eqref{eq:ith_comp_of_wedge} and by linearity of determinant, we get \eqref{eq:ith_comp_of_wedge_simplified}.
      
      By now, it remains to show that \eqref{eq:ith_comp_of_wedge_simplified} and \eqref{eq:(i+1)th_comp_of_wedge_in_rhs} are equal.
      
      Recall that
      $
         \mathcal{V} = {
            [ 
               \mathcal{A}\mathfrak{p}_{1}, 
               \mathcal{A}^{2}\mathfrak{p}_{1}, 
               \ldots, 
               \mathcal{A}^{n-2}\mathfrak{p}_{1}
            ]
         }
      $.
      If we put
      $
         \mathcal{U} = {
            [ 
               \mathfrak{q}^{2}, \mathfrak{q}^{3}, \cdots, 
               \hat{\mathfrak{q}}^{i+1}, \cdots, \mathfrak{q}^{n}
            ]^{t}
         }
      $, then \eqref{eq:ith_comp_of_wedge_simplified} is equal to $(-1)^{n-i+1}\det\mathcal{U}\mathcal{V}$. 
      By Lemma \ref{lemma:det_of_mul_of_irregular_mtx}, we have
      \begin{displaymath}
         \det \mathcal{U}\mathcal{V} = {
            \sum_{|S| = n-2} {
               \mathcal{U}_{\bar{T},S}\mathcal{V}_{S,\bar{T}}
            }
         },
      \end{displaymath}
      where $\bar{T} = [n-2]$.
      Note that $\mathcal{Q} = \adj \mathcal{P}$, the entries of $\mathcal{U}$ can be written explicitly:
      \begin{displaymath}
         \mathcal{U} = {
            \left[
            \begin{smallmatrix}
               -\mathcal{P}_{12} & \mathcal{P}_{22} & \cdots & (-1)^{n+2}\mathcal{P}_{n2} \\
               \vdots & \vdots & & \vdots \\
               \widehat{\mathcal{P}}_{1,i+1} & \widehat{\mathcal{P}}_{2,i+1} & \cdots & \widehat{\mathcal{P}}_{n,i+1} \\
               \vdots & \vdots & & \vdots \\
               (-1)^{n+1}\mathcal{P}_{1n} & (-1)^{n+2}\mathcal{P}_{2n} & \cdots & \mathcal{P}_{nn}
            \end{smallmatrix}
            \right]
         }.
      \end{displaymath}
      And if $S = [n] \backslash \{ k, l \}^{o}$, we can apply Proposition \ref{prop:det_compound_matrix} to get
      \begin{displaymath}
         \mathcal{U}_{\bar{T},S} = {
            (-1)^{k+l+i}
            \begin{vmatrix}
               p_{k1} & p_{l1} \\
               p_{k,i+1} & p_{l,i+1}
            \end{vmatrix}
         }.
      \end{displaymath}
      Hence we obtain
      \begin{align*}
         &\quad (-1)^{n-i+1}
         \left|[
            \mathcal{Q}^{(1,i+1)} \mathcal{A}\mathfrak{p}_1,
            \mathcal{Q}^{(1,i+1)} \mathcal{A}^{2}\mathfrak{p}_1,
            \ldots,
            \mathcal{Q}^{(1,i+1)} \mathcal{A}^{n-2}\mathfrak{p}_1
         ]\right| \\
         &= {
            (-1)^{n-i+1}\det\mathcal{U}\mathcal{V}
         } \\
         &= {
            (-1)^{n-i+1}
	         \sum_{|S| = n-2} {
	            \mathcal{U}_{\bar{T},S}\mathcal{V}_{S,\bar{T}}
            }
         } \\
         &= {
            (-1)^{n-i+1}
            \sum_{
               \substack {
                  S = [n] \backslash \{ k, l \}^{o} \\
                  S' = \{ k, l \}^{o}
               }
            } {
               (-1)^{k+l+i}
               \begin{vmatrix}
                  p_{k1} & p_{k,i+1} \\
                  p_{l1} & p_{l,i+1}
               \end{vmatrix}
               \mathcal{V}_{S,\bar{T}}
            }
         },
      \end{align*}
      which concludes that \eqref{eq:(i+1)th_comp_of_wedge_in_rhs} and \eqref{eq:ith_comp_of_wedge_simplified} are equal.
      Thus the $(i+1)$th component of 
      $
         \wedge^{n-1}[\mathfrak{p}_{1}, \mathcal{A}\mathfrak{p}_{1}, \mathcal{A}^{2}\mathfrak{p}_{1}, \ldots, \mathcal{A}^{n-2}\mathfrak{p}_{1}]
         \cdot \mathcal{P}
	   $
      is equal to the $i$th component of $\wedge^{n-2} [ \mathfrak{q}, \mathcal{B}\mathfrak{q}, \ldots, \mathcal{B}^{n-3}\mathfrak{q} ]$.
      The proof is completed.
   \end{step}
   \begin{step}
      We shall show that
      \begin{displaymath}
         {
            \left| \begin{smallmatrix}
               0 & u_1 & \cdots & u_{n} \\
               p_{11} & a_{11} & \cdots & a_{1n} \\
               \vdots & \vdots & & \vdots \\
               p_{n1} & a_{n1} & \cdots & a_{nn}
            \end{smallmatrix} \right|
         } = -{
            \left| \begin{smallmatrix}
               0 & w_1 & \cdots & w_{n-1} \\
               b_{21} & b_{22} & \cdots & b_{2n} \\
               \vdots & \vdots & & \vdots \\
               b_{n1} & b_{n2} & \cdots & b_{nn}
            \end{smallmatrix} \right|
         }.
      \end{displaymath}

      By the arguments similar to those in Step 2 of Proposition \ref{prop:det_compound_matrix},
      we may assume $R = \mathbb{C}$ and $p_{n1} \not= 0$.
      Denote $i$\textnormal{th} row of $\mathcal{Q}$ by $\mathfrak{q}^{i}$ and
      $j$\textnormal{th} column of $\mathcal{P}$ by $\mathfrak{p}_{j}$,
      then by Step 1 we have
      \begin{equation} \label{eq:def_of_D}
         \mathcal{D} := {
            \left[ \begin{smallmatrix}
               0 & w_1 & \cdots & w_{n-1} \\
               b_{21} & b_{22} & \cdots & b_{2n} \\
               \vdots & \vdots & & \vdots \\
               b_{n1} & b_{n2} & \cdots & b_{nn}
            \end{smallmatrix} \right]
         }
         = {
            \left[ \begin{smallmatrix}
               { \langle u_1, \ldots, u_n \rangle }\mathcal{P} \\
               \mathfrak{q}^{2}\mathcal{A}\mathcal{P} \\
               \vdots \\
               \mathfrak{q}^{n}\mathcal{A}\mathcal{P}
            \end{smallmatrix} \right]
         }
         = {
            \left[ \begin{smallmatrix}
               { \langle u_1, \ldots, u_n \rangle } \\
               \mathcal{Q}^{(1)}\mathcal{A}
            \end{smallmatrix} \right]\mathcal{P}
         },
      \end{equation}
      where $\mathcal{Q}^{(1)}$ is the submatrix of $\mathcal{Q}$ by deleting the first row.

      Put
      $
         \mathcal{S} = {
            \left[ \begin{smallmatrix}
               -p_{n1} & & & & p_{11} \\
               & -p_{n1} & & & p_{21} \\
               & & \ddots & & \vdots \\
               & & & -p_{n1} & p_{n-1,1} \\
            \end{smallmatrix} \right]
         }
      $ and $\mathcal{R} = \mathcal{S}\mathcal{P}$.
      Let $\mathcal{P}_{(1)}, \mathcal{R}_{(1)}$ be submatrices obtained by deleting the first column of $\mathcal{P}$ and $\mathcal{R}$, respectively.
      Then we get $\mathcal{R}_{(1)}=\mathcal{S}\mathcal{P}_{(1)}$.

      Note that the first column of $\mathcal{R}$ is zero, hence we have
      \begin{equation} \label{eq:equiv_of_dets_of_mul_of_submtxes}
         \mathcal{R}\mathcal{Q} = \mathcal{R} {
            \left[ \begin{smallmatrix}
               \mathfrak{q}^{1} \\ \vdots \\ \mathfrak{q}^{n}
            \end{smallmatrix} \right]
         }
         = \mathcal{R}_{(1)}{
            \left[ \begin{smallmatrix}
               \mathfrak{q}^{2} \\ \vdots \\ \mathfrak{q}^{n}
            \end{smallmatrix} \right]
         } = \mathcal{R}_{(1)}\mathcal{Q}^{(1)}.
      \end{equation}
      Thus from \eqref{eq:def_of_D}, \eqref{eq:equiv_of_dets_of_mul_of_submtxes} and definition of $\mathcal{R}$, we get
      \begin{equation} \label{eq:org_SA}
         {
            \left[ \begin{smallmatrix}
               1 & \\
               & \mathcal{R}_{(1)}
            \end{smallmatrix} \right] \mathcal{D}
         }
         = {
            \left[ \begin{smallmatrix}
               { \langle u_1, \ldots, u_n \rangle } \\
               \mathcal{R}_{(1)}\mathcal{Q}^{(1)}\mathcal{A}
            \end{smallmatrix} \right]\mathcal{P}
         }
         = {
            \left[ \begin{smallmatrix}
               { \langle u_1, \ldots, u_n \rangle } \\
               \mathcal{R}\mathcal{Q}\mathcal{A}
            \end{smallmatrix} \right]\mathcal{P}
         }
         = {
            \left[ \begin{smallmatrix}
               { \langle u_1, \ldots, u_n \rangle } \\
               \mathcal{S}\mathcal{A}
            \end{smallmatrix} \right]\mathcal{P}
         }.
      \end{equation}
      By Lemma \ref{lemma:minor_after_row_op} and $\det \mathcal{P} = 1$, we have 
      \begin{displaymath}
         \det \mathcal{R}_{(1)} = {
            \det \mathcal{S}\mathcal{P}_{(1)}
         } 
         = {
            p_{n1}^{n-2} \det \mathcal{P}
         }
         = p_{n1}^{n-2}, 
      \end{displaymath}
      and by \eqref{eq:org_SA} 
      \begin{displaymath}
         \det \mathcal{D} = {
            \frac{1}{p_{n1}^{n-2}}
            \begin{vmatrix}
               { \langle u_1, \ldots, u_n \rangle } \\
               \mathcal{S}\mathcal{A}
            \end{vmatrix}
         }.
      \end{displaymath}
      Expand along the first row, we obtain
      \begin{displaymath}
         \det \mathcal{D} = {
            \dfrac{1}{p_{n1}^{n-2}}
            \sum_{i=1}^{n} {
               (-1)^{i-1} u_{i} \det \mathcal{S}\mathcal{A}_{(i)}
            }
         }.
      \end{displaymath}
      Moreover, by Lemma \ref{lemma:minor_after_row_op}, we have
      \begin{displaymath}
         \det \mathcal{S}\mathcal{A}_{(i)} = p_{n1}^{n-2} {
            \left| \begin{smallmatrix}
               p_{11} & a_{11} & \cdots & \hat{a_{1i}} & \cdots & a_{1n} \\
               \vdots & \vdots & & \vdots & & \vdots \\
               p_{n1} & a_{n1} & \cdots & \hat{a_{ni}} & \cdots & a_{nn} \\
            \end{smallmatrix} \right|
         }.
      \end{displaymath}
      Hence we get
      \begin{displaymath}
         \det \mathcal{D} = -{
            \left| \begin{smallmatrix}
               0 & u_1 & \cdots & u_n \\
               p_{11} & a_{11} & \cdots & a_{1n} \\
               \vdots & \vdots & & \vdots \\
               p_{n1} & a_{n1} & \cdots & a_{nn}
            \end{smallmatrix} \right|
         }.
      \end{displaymath}
   \end{step}
\end{proof}

   \section{Main Theorem} \label{ch:noethers_problem}
In this section, we prove the following Main Theorem about rationality of $K(C_m \rtimes C_n)$.

\begin{main_thm}
	Let $m, n$ be positive integers, where $n$ is an odd prime, $K$ be a field such that neither $m$ nor $n$ is multiple of the characteristic of $K$ and both the primitive roots of unity $\zeta_m, \zeta_n$ lie in $K$.
	Let
	\begin{displaymath}
		G = C_m \rtimes_r C_n = \langle \sigma_1, \sigma_2 : \sigma_1^m = \sigma_2^n = 1, \sigma_2^{-1}\sigma_1\sigma_2 = \sigma_1^r \rangle,
	\end{displaymath}
	where $r^{n} \equiv 1 \pmod{m}$. 
	Let $m' = \frac{m}{\gcd (m, r-1)}$. 
	Suppose there exist relatively prime integers $a_1, \alpha_{n-2}, \ldots, \alpha_0 \in \mathbb{Z}$ such that
	\begin{displaymath}
		a_1m' = \alpha_{n-2} r^{n-2} + \alpha_{n-3} r^{n-3} + \cdots + \alpha_1 r + \alpha_0
	\end{displaymath}
	and
	\begin{displaymath}
		x = \alpha_{n-2}\zeta_{n}^{n-2} + \alpha_{n-3}\zeta_{n}^{n-3} + \cdots + \alpha_1\zeta_{n} + \alpha_0
		   \in \mathbb{Z}[\zeta_{n}]
	\end{displaymath}
	satisfies the norm $N_{\mathbb{Q}(\zeta_{n})/\mathbb{Q}}(x) = m'$.
	Then $K(G)$ is rational over $K$.
\end{main_thm}

The proof comprises three parts.
In \ref{sec:problem_reduction}, we reduce the group action to a monomial action with the corresponding matrix
\begin{displaymath}
   \Delta = {
      \left[ \begin{smallmatrix}
         r & & & & & -\frac{x_{n-1}}{m'} \\
         m' & & & & & -x_{n-2} \\
         & 1 & & & & -x_{n-3} \\
         & & \ddots & & & \vdots \\
         & & & 1 & 0 & -x_2 \\
         & & & 0 & 1 & -x_1
      \end{smallmatrix} \right]
   } \in GL_{n-1}(\mathbb{Z}),
\end{displaymath}
where $x_{j} = \frac{r^{j + 1} - 1}{r - 1}$ for $1 \leq j \leq n - 1$.
We then show that if $\Delta$ is conjugate to a matrix of the form
\begin{displaymath}
   \Gamma = {
      \left[ \begin{smallmatrix}
         \begin{smallmatrix}
            a & \\ 1 & \\ & 1
         \end{smallmatrix}
         & \textnormal{\Large *} \\
         0 &
         \begin{smallmatrix}
            \ddots & & \\ & 1 & b
         \end{smallmatrix}
      \end{smallmatrix} \right]
   }
\end{displaymath}
in $GL_{n-1}(\mathbb{Z})$, then $K(G)$ is rational over $K$.

In \ref{sec:algorithm}, we give an algorithm to transform $\Delta$ to a matrix of the form
\begin{displaymath}
   \left[ \begin{smallmatrix}
      \begin{smallmatrix}
         a & \\ e_1 & \\ & e_2
      \end{smallmatrix}
      & \textnormal{\Large *} \\
      0 &
      \begin{smallmatrix}
         \ddots & & \\ & e_{n-2} & b
      \end{smallmatrix}
   \end{smallmatrix} \right].
\end{displaymath}

In \ref{sec:val_of_most_important_entry}, we show that 
$e_{n-2} = \frac{N_{\mathbb{Q}(\zeta_{n})/\mathbb{Q}}(x)}{m' e_{n-3}^2 e_{n-4}^3 \ldots e_1^{n-2}}$, 
in which $x \in \mathbb{Z}[\zeta_n]$ is chosen in \ref{sec:algorithm} satisfying 
$N_{\mathbb{Q}(\zeta_n)/\mathbb{Q}}(x) = m'$, and complete the proof by deducing that $e_1 = e_2 = \cdots = e_{n-2} = 1$.

\subsection{} \label{sec:problem_reduction}
Let the group $G$ and the field $K$ be as in the Main Theorem.
A faithful representation of $G$ on $K^n$ is given by
\begin{displaymath}
   \sigma_1 \mapsto {
      \left[ \begin{smallmatrix}
         \zeta & & & & \\
         & \zeta^r & & & \\
         & & \zeta^{r^2} & & \\
         & & & \ddots & \\
         & & & & \zeta^{r^{n-1}}
      \end{smallmatrix} \right]
   },
   \sigma_2 \mapsto {
      \left[ \begin{smallmatrix}
         0 & & & & 1 \\
         1 & & & & \\
         & 1 & & & \\
         & & \ddots & & \\
         & & & 1 & 0 \\
      \end{smallmatrix} \right]
   },
\end{displaymath}
where $\zeta = \zeta_m$.
This representation induces an action of $G$ on $K(X_1, X_2, \ldots, X_n)$:
\begin{align*}
   &\sigma_1 : X_1 \mapsto \zeta X_1 \text{, } X_2 \mapsto \zeta^r X_2 \text{, } \ldots \text{, }
              X_{n-1} \mapsto \zeta^{r^{n-2}} X_{n-1} \text{, } X_n \mapsto \zeta^{r^{n-1}} X_n; \\
   &\sigma_2 : X_1 \mapsto X_2 \mapsto X_3 \mapsto \cdots \mapsto X_{n-1} \mapsto X_n \mapsto X_1.
\end{align*}
The fixed subfield is
\begin{displaymath}
   K(X_1, \ldots, X_n)^G = \{ h \in K(X_1, \ldots, X_n) : \sigma h = h, \forall \sigma \in G \}.
\end{displaymath}
By Corollary \ref{cor:redu_cor}, if $K(X_1, \ldots, X_n)^G$ is rational over $K$ then $K(G)$ is rational over $K$.

Define $Y_1 = X_1$, $Y_2 = X_2/X_1, Y_3 = X_3/X_2, \ldots, Y_n = X_n/X_{n-1}$,
then the action of $G$ on $K(Y_1, \ldots, Y_n) = K(X_1, \ldots, X_n)$ is
\begin{align*}
   &\sigma_1 : Y_1 \mapsto \zeta Y_1 \text{, } Y_2 \mapsto \zeta^{r-1} Y_2 \text{, } \ldots \text{, }
              Y_{n-1} \mapsto \zeta^{r^{n-2}-r^{n-3}} Y_{n-1} \text{, } Y_n \mapsto \zeta^{r^{n-1}-r^{n-2}} Y_n; \\
   &\sigma_2 : Y_1 \mapsto Y_1 Y_2 \text{, } Y_2 \mapsto Y_3 \mapsto Y_4 \mapsto \cdots \mapsto Y_{n} \mapsto \dfrac{1}{Y_2 Y_3 \cdots Y_n}.
\end{align*}
Note that the action is linear on $Y_1$ with coefficients in $K(Y_2, \ldots, Y_n)$, so by Theorem \ref{thm:AHK},
\begin{equation} \label{eq:reduce_of_fixed_field}
   K(Y_1, Y_2, \ldots, Y_n)^G = K(Y_2, Y_3, \ldots, Y_n)^G(Y_1'),
\end{equation}
for some $Y_1'$. 
Let $m' = \frac{m}{\gcd(m, r-1)}$ and define
\begin{displaymath}
   Z_1 = Y_2^{m'}, Z_2 = Y_3/Y_2^r, Z_3 = Y_4/Y_3^r, \ldots, Z_{n-1} = Y_n/Y_{n-1}^r.
\end{displaymath}
Then $K(Y_2, Y_3, \ldots, Y_n)^{\langle \sigma_1 \rangle} = K(Z_1, Z_2, \ldots, Z_{n-1})$ and
\begin{equation} \label{eq:action_of_sigma_2_on_Z}
   \sigma_2 : Z_1 \mapsto Z_1^r Z_2^{m'}, Z_2 \mapsto Z_3 \mapsto Z_4 \mapsto \cdots \mapsto Z_{n-1} \mapsto
   \dfrac{1}{Z_1^{\frac{r^n-1}{m'(r-1)}} Z_2^{\frac{r^{n-1}-1}{r-1}} \cdots Z_{n-1}^{\frac{r^2-1}{r-1}}}.
\end{equation}

This action is a purely monomial action with the corresponding matrix $\Delta$,
\begin{displaymath}
   \Delta = {
      \left[ \begin{smallmatrix}
         r & & & & & -\frac{x_{n-1}}{m'} \\
         m' & & & & & -x_{n-2} \\
         & 1 & & & & -x_{n-3} \\
         & & \ddots & & & \vdots \\
         & & & 1 & 0 & -x_2 \\
         & & & 0 & 1 & -x_1
      \end{smallmatrix} \right]
   } \in GL_{n-1}(\mathbb{Z}), 
\end{displaymath}
where
\begin{equation} \label{def:x_j}
   x_j = \frac{r^{j+1} - 1}{r - 1}
\end{equation}
for $1 \leq j \leq n-1$.
It's worth to note that $\Delta^{n} = I_{n-1}$ and $\det \Delta = (-1)^{n-1}$.

\begin{lemma} \label{cor:similarity_cor}
   If $\Delta$ is conjugate to
   \begin{displaymath}
      \Gamma =
      \left[ \begin{smallmatrix}
         \begin{smallmatrix}
            a & \\ 1 & \\ & 1
         \end{smallmatrix}
         & \textnormal{\Large *} \\
         0 &
         \begin{smallmatrix}
            \ddots & & \\ & 1 & b
         \end{smallmatrix}
      \end{smallmatrix} \right]
   \end{displaymath}
   in $GL_{n-1}(\mathbb{Z})$, then $K(G)$ is rational over $K$.
\end{lemma}
\begin{proof}
   \setcounter{step}{0}
   \begin{step}
      We shall first show that $\Gamma$ is conjugate to the matrix
      \begin{displaymath}
         \Sigma = 
         \left[
         \begin{smallmatrix}
            0 & & & & & -1 \\
            1 & & & & & -1 \\
            & 1 & & & & -1 \\
            & & \ddots & & & \vdots \\
            & & & 1 & & -1 \\
            & & & & 1 & -1 \\
         \end{smallmatrix}
         \right].
      \end{displaymath}
      Below, we shall show how to find a matrix $\mathcal{P} = \mathcal{P}_{n-1} \cdots \mathcal{P}_{1}$ such that $\mathcal{P}^{-1}\Gamma\mathcal{P} = \Sigma$.

      Set
      $
      \mathcal{P}_{n-1} = \left[{
         \begin{smallmatrix}
            1 & & & & \\
            & 1 & & & \\
            & & \ddots & & \\
            & & & 1 & -b-1 \\
            & & & 0 & 1
         \end{smallmatrix}
      }\right]
      $, then
      \begin{displaymath}
         \mathcal{N}_{n-1} := \mathcal{P}_{n-1}^{-1}\Gamma\mathcal{P}_{n-1} = \left[{
            \begin{smallmatrix}
               a & & & & \\
               1 & & & \textnormal{\Large *} & \\
               & 1 & & & \\
               & & \ddots & & \\
               & \textnormal{\Large 0} & & 1 & -1
            \end{smallmatrix}
         }\right].
      \end{displaymath}
      Now, suppose we have
      \begin{displaymath}
         \mathcal{N}_{i+1} = \left[{
            \begin{smallmatrix}
               a & & & & & & \\
               1 & & & & \text{\Large *} & & \\
               & \ddots & & & & & \\
               & & 1 & a_i & a_{i+1} & \cdots & a_{n-1} \\
               & & & 1 & 0 & & -1 \\
               & \textnormal{ \Large 0 } & & & \ddots & & \vdots \\
               & & & & & 1 & -1
            \end{smallmatrix}
         }\right].
      \end{displaymath}
      Conjugate $\mathcal{N}_{i+1}$ with the matrix
      \begin{displaymath}
         \mathcal{P}_i = \left[{
            \begin{smallmatrix}
               1 & & & & & \\
               & \ddots & & & & \\
               & & 1 & -a_i & \cdots & -a_{n-1} - 1 \\
               & & & 1 & & \\
               & & & & \ddots & \\
               & & & & & 1 \\
            \end{smallmatrix}
         }\right],
      \end{displaymath}
      we get
      \begin{displaymath}
         \mathcal{N}_i := \mathcal{P}_i^{-1}\mathcal{N}_{i+1}\mathcal{P}_i = \left[{
            \begin{smallmatrix}
               a & & & & & & \\
               1 & & & & \text{\Large *} & & \\
               & \ddots & & & & & \\
               & & 1 & b_{i-1} & b_i & \cdots & b_{n-1} \\
               & & & 1 & 0 & \cdots & -1 \\
               & \textnormal{ \Large 0 } & & & \ddots & & \vdots \\
               & & & & & 1 & -1 \\
            \end{smallmatrix}
         }\right]
      \end{displaymath}
      for some $b_{j}$.
      Repeat the process, then $\Delta$ is conjugate to
      $
         \mathcal{N} = \left[{
            \begin{smallmatrix}
               a_1 & a_2 & \cdots & a_{n-2} & a_{n-1} \\
               1 & & & & -1 \\
               & 1 & & & -1 \\
               & & \ddots & & \vdots \\
               & & & 1 & -1
            \end{smallmatrix}
         }\right]
      $ in $GL_{n-1}(\mathbb{Z})$, for some $a_j$. Note that $I_{n-1} = \Delta^{n} = \mathcal{N}^{n}$,
      regard $\mathcal{N}$ as the linear transformation:
      \begin{displaymath}
         \begin{tabular}{r c l}
            $\mathfrak{e}_1$ & $\mapsto$ & $a_1\mathfrak{e}_1 + a_2\mathfrak{e}_2 + \cdots + a_{n-2}\mathfrak{e}_{n-2} + a_{n-1}\mathfrak{e}_{n-1}$ \\
            $\mathfrak{e}_2$ & $\mapsto$ & $\mathfrak{e}_1 - \mathfrak{e}_{n-1}$ \\
            $\mathfrak{e}_3$ & $\mapsto$ & $\mathfrak{e}_2 - \mathfrak{e}_{n-1}$ \\
            & $\vdots$ &  \\
            $\mathfrak{e}_{n-1}$ & $\mapsto$ & $\mathfrak{e}_{n-2} - \mathfrak{e}_{n-1}$ \\
         \end{tabular}
      \end{displaymath}
      where $\{ \mathfrak{e}_{i} \}$ is the standard basis of $K^{n-1}$.
      Then the action of $\mathcal{N}^{n}$ on $\mathfrak{e_{n-1}}$ is:
      \begin{displaymath}
            \begin{tabular}{r c l}
            $\mathfrak{e}_{n-1}$ & $\mapsto$ & $(a_1 - 1)(a_1\mathfrak{e}_1 + a_2\mathfrak{e}_2 + \cdots + a_{n-2}\mathfrak{e}_{n-2} + a_{n-1}\mathfrak{e}_{n-1})$ \\
            & & $\quad + a_2(\mathfrak{e}_1 - \mathfrak{e}_{n-1}) + \cdots + a_{n-2}(\mathfrak{e}_{n-3} - \mathfrak{e}_{n-1}) + (a_{n-1} + 1)(\mathfrak{e}_{n-2} - \mathfrak{e}_{n-1})$. \\
            \end{tabular}
      \end{displaymath}
      Write $\beta = a_1 - 1$, then
      \begin{align*}
         \mathfrak{e}_{n-1}
         &= (\beta a_1 + a_2)\mathfrak{e}_1 + (\beta a_2 + a_3)\mathfrak{e}_2 + \cdots +
         (\beta a_{n-2} + a_{n-1} + 1)\mathfrak{e}_{n-2} \\
         &\qquad +(\beta a_{n-1} - a_2 - \cdots - a_{n-1} - 1)\mathfrak{e}_{n-1},
      \end{align*}
      which gives us a system of equations.
      Solve it to get
      \begin{align*}
         a_{2} &= -\beta a_{1} \\
         a_{3} &= \beta^{2} a_{1} \\
         \quad &\cdots \\
         a_{i+1} &= (-1)^{i} \beta^{i} a_{1} \\
         \quad &\cdots \\
         a_{n-2} &= (-1)^{n-3}\beta^{n-3} a_1 \\
         a_{n-1} &= (-1)^{n-2}\beta^{n-2} a_1 - 1 \\
      \end{align*}                  
      and
      \begin{align*}
         1 &= \beta a_{n-1} - (a_2 + a_3 + \cdots + a_{n-2} + a_{n-1} + 1) \\
         &= \beta a_{n-1} + \beta a_1 - \beta^2 a_1 + \cdots + (-1)^{n-2}\beta^{n-3}a_1 + (-1)^{n-1}\beta^{n-2}a_1 \\
         &= [(-1)^{n-2} \beta^{n-1} a_1 - \beta] + \beta a_1 \dfrac{(-\beta)^{n-2} - 1}{(-\beta) - 1} \\
         &= (-1)^{n-2}(a_1 - 1)^{n}.
      \end{align*}

      Since $n$ is odd, we get $a_1 = a_2 = \cdots = a_{n-2} = 0$ and $a_{n-1} = -1$. Hence $\Delta$ is conjugate to the matrix $\Sigma$.

%
   \end{step}
   \begin{step}
      We shall show that $K(X_1, X_2, \ldots, X_{n-1}, X_n)^G$ is rational over $K$.

      From \eqref{eq:reduce_of_fixed_field}, \eqref{eq:action_of_sigma_2_on_Z} and Step 1, we may assume that
      \begin{displaymath}
            K(X_1, \ldots, X_n)^G = K(Y_1, \ldots, Y_{n-1})^{\langle \sigma_2 \rangle}(Z),
      \end{displaymath}
      where the action of $\sigma_2$ on $Y_1, \ldots, Y_{n-1}$ is given by
      \begin{displaymath}
         \sigma_2 : Y_1 \mapsto Y_2 \mapsto \cdots \mapsto Y_{n-2} \mapsto Y_{n-1} \mapsto \dfrac{1}{Y_1Y_2 \cdots Y_{n-1}}.
      \end{displaymath}
      Define
      \begin{displaymath}
         Z_i = \dfrac{1 + \zeta_n^i Y_1 + \zeta_n^{2i} Y_1Y_2 + \cdots + \zeta_n^{i(n-1)} Y_1 \cdots Y_{n-1}}{1 + Y_1 + Y_1Y_2 + \cdots + Y_1Y_2 \cdots Y_{n-1}}, 1 \leq i \leq n-1.
      \end{displaymath}
      Then
      \begin{displaymath}
         \sigma_2 : Z_1 \mapsto \zeta_n^{n-1}Z_1, Z_2 \mapsto \zeta_n^{2(n-1)}Z_2, \ldots, Z_{n-1} \mapsto \zeta_n^{(n-1)^2}Z_{n-1}.
      \end{displaymath}
      It is a linear action and $\zeta_n \in K$, so by Theorem \ref{thm:abel_case}, $K(Z_1, \ldots, Z_{n-1})^{\langle \sigma_2 \rangle}$ is rational.

      Moreover,
      \begin{displaymath}
            Z_1 + Z_2 + \cdots + Z_{n-1} = \dfrac{n - (1 + Y_1 + Y_1Y_2 + \cdots + Y_1Y_2 \cdots Y_{n-1})}{1 + Y_1 + Y_1Y_2 + \cdots + Y_1Y_2 \cdots Y_{n-1}}.
      \end{displaymath}
      Thus
      \begin{align*}
         &\quad K(Z_1, Z_2, \ldots, Z_{n-1}) \\
         &= K(1 + \zeta_n^iY_1 + \zeta_n^{2i}Y_1Y_2 + \cdots + \zeta_n^{i(n-1)}Y_1Y_2 \cdots Y_{n-1} : 0 \leq i \leq n-1) \\
         &= K(Y_1, Y_1Y_2, \ldots, Y_1Y_2 \cdots Y_{n-1}) \\
         &= K(Y_1, Y_2, \ldots, Y_{n-1}).
      \end{align*}
      Hence
      \begin{align*}
         K(X_1, X_2, \ldots, X_{n-1}, X_n)^G
         &= K(Y_1, Y_2, \ldots, Y_{n-1})^{\langle \sigma_2 \rangle}(Z) \\
         &= K(Z_1, \ldots, Z_{n-1})^{\langle \sigma_2 \rangle}(Z)
      \end{align*}
      is rational over $K$.
   \end{step}
\end{proof}

\subsection{} \label{sec:algorithm}
In this section, we describe the algorithm to transform $\Delta$ to a matrix of the form
\begin{displaymath}
   \left[ \begin{smallmatrix}
      \begin{smallmatrix}
         a & \\ e_1 & \\ & e_2
      \end{smallmatrix}
      & \textnormal{\Large *} \\
      0 &
      \begin{smallmatrix}
         \ddots & & \\ & e_{n-2} & b
      \end{smallmatrix}
   \end{smallmatrix} \right].
\end{displaymath}
We first list two simple lemmas, the proofs are straightforward and omitted.

\begin{lemma} \label{lemma:rel_prime_sln_matrix}
   Let $a_1, \ldots, a_n \in \mathbb{Z}$ be relatively prime.
   There exists $\mathcal{M} \in SL_n(\mathbb{Z})$ with first column $\langle a_1, \ldots, a_n \rangle^t$.
\end{lemma}

\begin{lemma} \label{lemma:conj_result}
   Let $\mathcal{M}$ be a matrix of the form
   $
      \left[ \begin{smallmatrix}
         \mathcal{A} & \mathcal{C} \\
         \mathcal{B} & \mathcal{D}
      \end{smallmatrix} \right]
   $,
   where $\mathcal{A} \in M_k(\mathbb{Z})$, $\mathcal{D} \in M_l(\mathbb{Z})$.
   The first $(k-1)$ columns of $\mathcal{B}$ are zero and the last column is $\langle b_1, b_2, \ldots, b_l \rangle^t$.
   Assume that $d = \gcd \{ b_1, \ldots, b_l \} \not= 0$ and $b_j' = b_j / d$.
   Choose any matrix $\mathcal{P}' \in SL_l(\mathbb{Z})$ such that the first column is $\langle b_1', b_2', \ldots, b_l' \rangle^t$.
   Set
   $
      \mathcal{P} = {
      \left[ \begin{smallmatrix}
         I_k & \\
         & \mathcal{P}'
         \end{smallmatrix} \right]
      }
   $ and
   $
      \mathcal{P}^{-1}\mathcal{M}\mathcal{P} = {
         \left[ \begin{smallmatrix}
            \mathcal{A}' & \mathcal{C}' \\
            \mathcal{B}' & \mathcal{D}'
         \end{smallmatrix} \right]
      }
   $.
   Then $\mathcal{D}' = (\mathcal{P}')^{-1}\mathcal{D}\mathcal{P}'$ 
   and all the entries of $\mathcal{B}'$ are zero except that the $(1, k)$ entry is $d$.
\end{lemma}

Suppose we have
\begin{equation} \label{eq:x_in_r}
   a_1m' = \alpha_{n-2} r^{n-2} + \alpha_{n-3} r^{n-3} + \cdots + \alpha_1 r + \alpha_0,
\end{equation}
where $a_1 \in \mathbb{N}, \alpha_i \in \mathbb{Z}$. We may assume
$\gcd \{ a_1, \alpha_{n-2}, \alpha_{n-3}, \ldots, \alpha_1, \alpha_0 \} = 1$.
We define
\begin{equation} \label{eq:def_of_a_i}
   \begin{tabular}{r c l}
      $a_{n-1}$ & $=$ & $\alpha_{n-2}$ \\
      $a_{n-2}$ & $=$ & $\alpha_{n-2} r + \alpha_{n-3}$ \\
      & $\vdots$ & \\
      $a_i$ & $=$ & $\alpha_{n-2} r^{n-i-1} + \alpha_{n-3} r^{n-i-2} + \cdots + \alpha_{i-1}$ \\
      & $\vdots$ & \\
      $a_2$ & $=$ & $\alpha_{n-2} r^{n-3} + \alpha_{n-3} r^{n-4} + \cdots + \alpha_2 r + \alpha_1$. \\
   \end{tabular}
\end{equation}
Then $\gcd \{ a_1, a_2, \ldots, a_{n-1} \} = 1$.
Thus by Lemma \ref{lemma:rel_prime_sln_matrix}, there is a matrix $\mathcal{P}_0 \in SL_{n-1}(\mathbb{Z})$ with the first column $\langle a_1, a_2, \ldots, a_{n-1} \rangle^t$.

Let $\mathcal{B}_1 = [b_{ij}^{(1)}] = \mathcal{P}_0^{-1} \Delta \mathcal{P}_0$.
Note that if $\langle b_{21}^{(1)}, b_{31}^{(1)}, \ldots, b_{n-1,1}^{(1)} \rangle$ is zero, 
then the matrix $\mathcal{B}_{1}$ is reducible.
Since $n$ is prime, the minimal polynomial of $\Delta$ is $X^{n-1} + \cdots + X + 1$,
which is irreducible, a contradiction. Hence at least one of $b_{i1}^{(1)}$ is not zero.
Let $e_1 = \gcd \{ b_{21}^{(1)}, b_{31}^{(1)}, \ldots, b_{n-1,1}^{(1)} \}  \not= 0$ and put
$b_{k1}^{'(1)} = b_{k1}^{(1)} / e_1$ for $2 \leq k \leq n-1$. Then we have
\begin{displaymath}
   \gcd \{ b_{21}^{'(1)}, b_{31}^{'(1)}, \ldots, b_{n-1,1}^{'(1)} \} = 1.
\end{displaymath}
By Lemma \ref{lemma:rel_prime_sln_matrix} again, there is a matrix
$
   \mathcal{P}_1 = {
      \left[ \begin{smallmatrix}
         1 & \\
          & \mathcal{P}_1'
      \end{smallmatrix} \right]
   },
$
where $\mathcal{P}_1' \in SL_{n-2}(\mathbb{Z})$ with the first column
$\langle b_{21}^{'(1)}, b_{31}^{'(1)}, \ldots, b_{n-1,1}^{'(1)} \rangle^t$.

Let $\mathcal{B}_2 = [b_{ij}^{(2)}] = \mathcal{P}_1^{-1}\mathcal{B}_1P_1$, then by Lemma \ref{lemma:conj_result}, $\mathcal{B}_2$ has the form
\begin{displaymath}
   \mathcal{B}_2 = {
      \left[
         \begin{array}{c:ccc}
            { b_{11}^{(1)} } & \text{\Large *} & &  \\ \hdashline
            e_1 & b_{22}^{(2)} & & \\
            0 & b_{32}^{(2)} & & \\
            \vdots & \vdots & \text{\Large *} & \\
            0 & b_{n-1,2}^{(2)} & &
         \end{array}
      \right]
   }.
\end{displaymath}
As argument above, for $3 \leq k \leq n-1$, at least one of $b_{k2}^{(2)}$ is not zero.
Let $e_2 = \gcd \{ b_{32}^{(2)}, b_{42}^{(2)}, \ldots, b_{n-1,2}^{(2)} \} \not= 0$ and put $b_{k2}^{'(2)} = b_{k2}^{(2)} / e_2$. 
Then we have
\begin{displaymath}
   \gcd \{ b_{32}^{'(2)}, b_{42}^{'(2)}, \ldots, b_{n-1,2}^{'(2)} \} = 1.
\end{displaymath}
There is a matrix
$
   \mathcal{P}_2 = {
      \left[ \begin{smallmatrix}
         I_2 & \\
          & \mathcal{P}_2'
      \end{smallmatrix} \right]
   }
$, where
$\mathcal{P}_2' \in SL_{n-3}(\mathbb{Z})$ with the first column \\
$\langle b_{32}^{'(2)}, b_{42}^{'(2)}, \ldots, b_{n-1,2}^{'(2)} \rangle^t$.

Let $\mathcal{B}_3 = [b_{ij}^{(3)}] = \mathcal{P}_2^{-1}\mathcal{B}_2P_2$, then $\mathcal{B}_3$ has the form
\begin{displaymath}
   \mathcal{B}_3 = {
      \left[
         \begin{array}{cc:ccc}
            b_{11}^{(1)} & * & & &   \\
            e_1 & * & \text{ \Large *} & &  \\ \hdashline
            0 & e_2 & b_{33}^{(3)} & &  \\
            0 & 0 & b_{43}^{(3)} & \text{ \Large *} &  \\
            \vdots & \vdots & \vdots & &  \\
            0 & 0 & b_{n-1,3}^{(3)} & &
         \end{array}
      \right]
   }.
\end{displaymath}

Proceed repeatedly as above. At last, we obtain
\begin{align} \label{eq:def_of_2nd_last_B}
   \mathcal{B}_{n-3} &= \mathcal{P}_{n-4}^{-1}\mathcal{B}_{n-4}\mathcal{P}_{n-4} \nonumber \\
   &= {
      \left[
         \begin{array}{ccc:ccc}
            b_{11}^{(n-3)} & & & & & \\
            e_1 & & \text{\Large *} & & & \\
            & \ddots & & \text{ \Large *} & & \\ \hdashline
            & & e_{n-4} & b_{n-3, n-3}^{(n-3)} & b_{n-3, n-2}^{(n-3)} & b_{n-3, n-1}^{(n-3)} \\
            & \text{\Large 0} & & b_{n-2, n-3}^{(n-3)} & b_{n-2, n-2}^{(n-3)} & b_{n-2, n-1}^{(n-3)} \\
            & & & b_{n-1, n-3}^{(n-3)} & b_{n-1, n-2}^{(n-3)} & b_{n-1, n-1}^{(n-3)} \\
         \end{array}
      \right]
   }.
\end{align}
Choose $\alpha, \beta \in \mathbb{Z}$ such that $b_{n-2, n-3}^{'(n-3)}\beta - b_{n-1, n-3}^{'(n-3)}\alpha = 1$, then
\begin{equation} \label{eq:def_of_last_P}
   \mathcal{P}_{n-3} = {
      \left[
         \begin{array}{c:cc}
            I_{n-3} & & \\ \hdashline
            & b_{n-2, n-3}^{'(n-3)} & \alpha \\
            & b_{n-1, n-3}^{'(n-3)} & \beta
         \end{array}
      \right]
   } \in SL_{n-1}(\mathbb{Z})
\end{equation}
and
\begin{equation} \label{eq:def_of_last_B}
   \mathcal{B}_{n-2} = \mathcal{P}_{n-3}^{-1}\mathcal{B}_{n-3}\mathcal{P}_{n-3} = {
      \left[
         \begin{array}{ccc:ccc}
            b_{11}^{(n-2)} & & & & \\
            e_1 & & \text{\Large *} & & \\
            & \ddots & & \text{ \Large *} & \\ \hdashline
            & & e_{n-3} & b_{n-2, n-2}^{(n-2)} & b_{n-2, n-1}^{(n-2)} \\
            & \text{\Large 0} & & b_{n-1, n-2}^{(n-2)} & b_{n-1, n-1}^{(n-2)} \\
         \end{array}
      \right]
   }.
\end{equation}

   \subsection{} \label{sec:val_of_most_important_entry}
In this section, we shall show that $e_1 = \cdots = e_{n-3} = b_{n-1, n-2}^{(n-2)} = 1$.

We claim that, with notations as in previous section, we have
   \begin{displaymath}
      b_{n-1, n-2}^{(n-2)} = \frac{N_{\mathbb{Q}(\zeta_n)/\mathbb{Q}}(x)}{m' e_{n-3}^2 e_{n-4}^3 \ldots e_1^{n-2}},
   \end{displaymath}
   where $x = \alpha_{n-2}\zeta_{n}^{n-2} + \alpha_{n-3}\zeta_{n}^{n-3} + \cdots + \alpha_1\zeta_{n} + \alpha_0 \in \mathbb{Z}[\zeta_{n}]$.
   \setcounter{step}{0}
   \begin{step}
      We first show that $b_{n-1, n-2}^{(n-2)}$ is actually a determinant of a particular matrix \eqref{eq:val_of_most_important_entry}.

      From \eqref{eq:def_of_2nd_last_B}, \eqref{eq:def_of_last_P} and \eqref{eq:def_of_last_B}, we have
      \begin{displaymath}
         \left[ \begin{smallmatrix}
            \beta & -\alpha \\
            -b_{n-1, n-3}^{'(n-3)} & b_{n-2, n-3}^{'(n-3)}
         \end{smallmatrix} \right]
         \left[ \begin{smallmatrix}
            b_{n-2,n-2}^{(n-3)} & b_{n-2,n-1}^{(n-3)} \\
            b_{n-1,n-2}^{(n-3)} & b_{n-1,n-1}^{(n-3)}
         \end{smallmatrix} \right]
         \left[ \begin{smallmatrix}
            b_{n-2, n-3}^{'(n-3)} & \alpha \\
            b_{n-1, n-3}^{'(n-3)} & \beta
         \end{smallmatrix} \right]
         = {
            \left[ \begin{smallmatrix}
               b_{n-2,n-2}^{(n-2)} & b_{n-2,n-1}^{(n-2)} \\
               b_{n-1,n-2}^{(n-2)} & b_{n-1,n-1}^{(n-2)}
            \end{smallmatrix} \right]
         }.
      \end{displaymath}
      By direct computation, we get
      \begin{align} \label{eq:entry}
         b_{n-1, n-2}^{(n-2)} &= {
            \left| \begin{smallmatrix}
               0 & -b_{n-1,n-3}^{'(n-3)} & b_{n-2,n-3}^{'(n-3)} \\
               b_{n-2,n-3}^{'(n-3)} & b_{n-2,n-2}^{(n-3)} & b_{n-2,n-1}^{(n-3)} \\
               b_{n-1,n-3}^{'(n-3)} & b_{n-1,n-2}^{(n-3)} & b_{n-1,n-1}^{(n-3)}
            \end{smallmatrix} \right|
         } \nonumber \\
         &= {
            \frac{1}{e_{n-3}^2}
            \left| \begin{smallmatrix}
               0 & -b_{n-1,n-3}^{(n-3)} & b_{n-2,n-3}^{(n-3)} \\
               b_{n-2,n-3}^{(n-3)} & b_{n-2,n-2}^{(n-3)} & b_{n-2,n-1}^{(n-3)} \\
               b_{n-1,n-3}^{(n-3)} & b_{n-1,n-2}^{(n-3)} & b_{n-1,n-1}^{(n-3)}
            \end{smallmatrix} \right|
         }.
      \end{align}
      Note that
      $
         \langle -b_{n-1,n-3}^{(n-3)}, b_{n-2,n-3}^{(n-3)} \rangle = {
            \wedge^{1} \langle b_{n-2,n-3}^{(n-3)}, b_{n-1,n-3}^{(n-3)} \rangle^{t}
         }
      $, where $\wedge^{1} \langle b_{n-2,n-3}^{(n-3)}, b_{n-1,n-3}^{(n-3)} \rangle^{t}$ is defined in Definition \ref{def:wedge}, 
      and
      \begin{displaymath}
         \left[ \begin{smallmatrix}
            b_{n-3,n-3}^{(n-3)} & b_{n-3,n-2}^{(n-3)} & b_{n-3,n-1}^{(n-3)} \\
            b_{n-2,n-3}^{(n-3)} & b_{n-2,n-2}^{(n-3)} & b_{n-2,n-1}^{(n-3)} \\
            b_{n-1,n-3}^{(n-3)} & b_{n-1,n-2}^{(n-3)} & b_{n-1,n-1}^{(n-3)} \\
         \end{smallmatrix} \right]
         = {
            \mathcal{P}_{n-4}^{'-1}
            \left[ \begin{smallmatrix}
	            b_{n-3,n-3}^{(n-4)} & b_{n-3,n-2}^{(n-4)} & b_{n-3,n-1}^{(n-4)} \\
	            b_{n-2,n-3}^{(n-4)} & b_{n-2,n-2}^{(n-4)} & b_{n-2,n-1}^{(n-4)} \\
	            b_{n-1,n-3}^{(n-4)} & b_{n-1,n-2}^{(n-4)} & b_{n-1,n-1}^{(n-4)} \\
            \end{smallmatrix} \right]
            \mathcal{P}_{n-4}'
         },
      \end{displaymath}
      where $\mathcal{P}_{n-4}' \in SL_{3}(\mathbb{Z})$ with first column
      $\mathfrak{q}' := \langle b_{n-3,n-4}^{'(n-4)}, b_{n-2,n-4}^{'(n-4)}, b_{n-1,n-4}^{'(n-4)} \rangle^{t}$.

      Apply Proposition \ref{prop:minor_of_mtx_conjugation} to \eqref{eq:entry}, where we take 
      \begin{displaymath}
         \mathcal{B} = {
            \left[ \begin{smallmatrix}
	            b_{n-3,n-3}^{(n-4)} & b_{n-3,n-2}^{(n-4)} & b_{n-3,n-1}^{(n-4)} \\
	            b_{n-2,n-3}^{(n-4)} & b_{n-2,n-2}^{(n-4)} & b_{n-2,n-1}^{(n-4)} \\
	            b_{n-1,n-3}^{(n-4)} & b_{n-1,n-2}^{(n-4)} & b_{n-1,n-1}^{(n-4)} \\
            \end{smallmatrix} \right]
         }, 
      \end{displaymath}
      we get
      \begin{align} \label{eq:entry_4_by_4}
         b_{n-1, n-2}^{(n-2)} &= {
            -\frac{1}{e_{n-3}^{2}}
            \left| \begin{smallmatrix}
	               0 & {
	                  \wedge^{2} \left[
	                     \mathfrak{q}', \mathcal{B}\mathfrak{q}'
	                  \right]
	               } \\
	               {
	                  \begin{matrix}
	                     b_{n-3,n-4}^{'(n-4)} \\ b_{n-2,n-4}^{'(n-4)} \\ b_{n-1,n-4}^{'(n-4)}
	                  \end{matrix}
	               }
	               &
	               {
				         \begin{matrix}
				            b_{n-3,n-3}^{(n-4)} & b_{n-3,n-2}^{(n-4)} & b_{n-3,n-1}^{(n-4)} \\
				            b_{n-2,n-3}^{(n-4)} & b_{n-2,n-2}^{(n-4)} & b_{n-2,n-1}^{(n-4)} \\
				            b_{n-1,n-3}^{(n-4)} & b_{n-1,n-2}^{(n-4)} & b_{n-1,n-1}^{(n-4)} \\
				         \end{matrix}
	               }
            \end{smallmatrix} \right|
         }.
      \end{align}

      Write $\mathfrak{q} = \langle b_{n-3,n-4}^{(n-4)}, b_{n-2,n-4}^{(n-4)}, b_{n-1,n-4}^{(n-4)} \rangle^{t}$, then we have $\mathfrak{q} = e_{n-4}\mathfrak{q}'$.
      Thus, each column of the matrix $\left[ \mathfrak{q}, \mathcal{B}\mathfrak{q} \right]$ is $e_{n-4}$ times that of $\left[ \mathfrak{q}', \mathcal{B}\mathfrak{q}' \right]$.
      Therefore, we get
      \begin{equation} \label{eq:coef_of_wedges}
         \wedge^{2} \left[ \mathfrak{q}', \mathcal{B}\mathfrak{q}' \right] = {
            \frac{1}{e_{n-4}^{2}}
            \wedge^{2} \left[ \mathfrak{q}, \mathcal{B}\mathfrak{q} \right]
         }.
      \end{equation}

      From \eqref{eq:entry_4_by_4} and \eqref{eq:coef_of_wedges}, we obtain
      \begin{displaymath}
         b_{n-1, n-2}^{(n-2)} = {
            -\frac{1}{e_{n-3}^{2}e_{n-4}^{3}}
            \left| \begin{smallmatrix}
	               0 & {
	                  \wedge^{2} \left[
	                     \mathfrak{q}, \mathcal{B}\mathfrak{q}
	                  \right]
	               } \\
	               {
	                  \begin{smallmatrix}
	                     b_{n-3,n-4}^{(n-4)} \\ b_{n-2,n-4}^{(n-4)} \\ b_{n-1,n-4}^{(n-4)}
	                  \end{smallmatrix}
	               }
	               &
	               {
				         \begin{smallmatrix}
				            b_{n-3,n-3}^{(n-4)} & b_{n-3,n-2}^{(n-4)} & b_{n-3,n-1}^{(n-4)} \\
				            b_{n-2,n-3}^{(n-4)} & b_{n-2,n-2}^{(n-4)} & b_{n-2,n-1}^{(n-4)} \\
				            b_{n-1,n-3}^{(n-4)} & b_{n-1,n-2}^{(n-4)} & b_{n-1,n-1}^{(n-4)} \\
				         \end{smallmatrix}
	               }
            \end{smallmatrix} \right|
         }.
      \end{displaymath}

      Apply Proposition \ref{prop:minor_of_mtx_conjugation} repeatedly, at last, we get
      \begin{align} \label{eq:val_of_most_important_entry}
         b_{n-1, n-2}^{(n-2)}
         &= {
            \dfrac{(-1)^{n-1}}{e_{n-3}^{2} e_{n-4}^{3} \cdots e_{1}^{n-2}}
            \left|
            \begin{tabular}{ c:c }
               $0$ & {$
                  \wedge^{n-2}[\mathfrak{p}, \Delta\mathfrak{p}, \ldots, \Delta^{n-3}\mathfrak{p}]
               $} \\ \hdashline
               {$
                  \begin{matrix}
                     a_1 \\ a_2 \\ a_3 \\ \vdots \\ a_{n-1}
                  \end{matrix}
               $} &
               {$
                  \begin{matrix}
                     r & & & & -\frac{x_{n-1}}{m'} \\
                     m' & & & & -x_{n-2} \\
                     & 1 & & & -x_{n-3} \\
                     & & \ddots & & \vdots \\
                     & & & 1 & -x_{1}
                  \end{matrix}
               $}
            \end{tabular}
            \right|
         },
      \end{align}
      where $\mathfrak{p} = \langle a_1, a_2, \ldots, a_{n-1} \rangle^t$.
   \end{step}
   \begin{step}
      We shall compute $b_{n-1,n-2}^{(n-2)}$ in \eqref{eq:val_of_most_important_entry} more explicitly.

      Write $
         \mathcal{H} = {
               [\mathfrak{p}, \Delta\mathfrak{p}, \ldots, \Delta^{n-3}\mathfrak{p}]
         }
         = {
            \left[ \begin{smallmatrix} \mathfrak{h}^{1} \\ \mathfrak{h}^{2} \\ \vdots \\ \mathfrak{h}^{n-1} \end{smallmatrix} \right]
         }
      $, then we have
      \begin{align*}
      &\quad \wedge^{n-2}[\mathfrak{p}, \Delta\mathfrak{p}, \ldots, \Delta^{n-3}\mathfrak{p}] \\
      &= \left\langle {
         {
            (-1)^{n-2}
            \left| \begin{smallmatrix}
               \mathfrak{h}^2 \\ \mathfrak{h}^3 \\ \vdots \\ \mathfrak{h}^{n-2} \\ \mathfrak{h}^{n-1}
            \end{smallmatrix} \right|
         },
         {
            \frac{(-1)^{n-3}}{m'}
            \left| \begin{smallmatrix}
               m'\mathfrak{h}^{1} \\ \mathfrak{h}^3 \\ \vdots \\ \mathfrak{h}^{n-2} \\ \mathfrak{h}^{n-1}
            \end{smallmatrix} \right|
         },
         \ldots,
         {
            \frac{-1}{m'}
            \left| \begin{smallmatrix}
               m'\mathfrak{h}^{1} \\ \mathfrak{h}^2 \\ \vdots \\ \mathfrak{h}^{n-3} \\ \mathfrak{h}^{n-1}
            \end{smallmatrix} \right|
         },
         {
            \frac{1}{m'}
            \left| \begin{smallmatrix}
               m'\mathfrak{h}^{1} \\ \mathfrak{h}^2 \\ \vdots \\ \mathfrak{h}^{n-3} \\ \mathfrak{h}^{n-2}
            \end{smallmatrix} \right|
         }
      } \right\rangle.
      \end{align*}
      For each $2 \leq i \leq n-2$, the $i$th component of
      $\wedge^{n-2} [\mathfrak{p}, \Delta\mathfrak{p}, \ldots, \Delta^{n-3}\mathfrak{p}]$ is
      \begin{align*}
      d_i &:= {
         \frac{(-1)^{n-i-1}}{m'}
         \left| \begin{smallmatrix}
            m'\mathfrak{h}^{1} \\ \vdots \\ \mathfrak{h}^{i-1} \\ \widehat{\mathfrak{h}^{i}} \\ \mathfrak{h}^{i+1} \\ \vdots \\ \mathfrak{h}^{n-1}
         \end{smallmatrix} \right|
      }
      = {
         \frac{(-1)^{n-i-1}}{m'}
         \left| \begin{smallmatrix}
            m'\mathfrak{h}^{1}-r\mathfrak{h}^{2} \\ \vdots \\ \mathfrak{h}^{i-1}-r^2\mathfrak{h}^{i+1} \\ \widehat{\mathfrak{h}^{i}} \\ \mathfrak{h}^{i+1}-r\mathfrak{h}^{i+2} \\ \vdots \\ \mathfrak{h}^{n-2} - r\mathfrak{h}^{n-1} \\ \mathfrak{h}^{n-1}
         \end{smallmatrix} \right|
      }.
      \end{align*}
      Define
      $
         \mathcal{R} = {
            \left[ \begin{smallmatrix}
               m'\mathfrak{h}^{1}-r\mathfrak{h}^{2} \\ \mathfrak{h}^{2}-r\mathfrak{h}^{3} \\ \vdots \\ \mathfrak{h}^{n-2}-r\mathfrak{h}^{n-1} \\ \mathfrak{h}^{n-1}
            \end{smallmatrix} \right]
         }
      $ and denote the $i$th row of $\mathcal{R}$ by $\mathfrak{r}^{i}$, we get
      \begin{displaymath}
         d_i = {
	         \frac{(-1)^{n-i-1}}{m'}
	         \left| \begin{smallmatrix}
	            \mathfrak{r}^{1} \\ \vdots \\ \mathfrak{r}^{i-1}+r\mathfrak{r}^{i} \\ \mathfrak{r}^{i+1} \\ \vdots \\ \mathfrak{r}^{n-1}
	         \end{smallmatrix} \right|
         }.
      \end{displaymath}
      Similarly, the last component is
      \begin{displaymath}
      d_{n-1} := \dfrac{1}{m'}\left| \begin{smallmatrix} \mathfrak{h}^{1} \\ \vdots \\ \mathfrak{h}^{n-3} \\ \mathfrak{h}^{n-2} \end{smallmatrix} \right|
          = \dfrac{1}{m'}\left| \begin{smallmatrix} \mathfrak{r}^{1} \\ \vdots \\ \mathfrak{r}^{n-3} \\ \mathfrak{h}^{n-2} \end{smallmatrix} \right|
          = \dfrac{1}{m'}\left| \begin{smallmatrix} \mathfrak{r}^{1} \\ \vdots \\ \mathfrak{r}^{n-3} \\ \mathfrak{r}^{n-2} + r\mathfrak{r}^{n-1} \end{smallmatrix} \right|
      \end{displaymath}
      and the first component is
      $d_1 := (-1)^{n-2}\left| \begin{smallmatrix} \mathfrak{r}^{2} \\ \vdots \\ \mathfrak{r}^{n-1} \end{smallmatrix} \right|$.

      For each $1 \leq i \leq n-1$, let $\omega_i$ be minor of $\mathcal{R}$ by deleting $i$th row, then we obtain
      \begin{equation} \label{eq:1st_and_last_of_d_i_in_R}
         d_1 = (-1)^{n-2}\omega_{1},
         \qquad
         d_{n-1} = \frac{1}{m'}(\omega_{n-1} + r\omega_{n-2})
      \end{equation}
       and
      \begin{equation} \label{eq:d_i_in_R}
         d_{i} = \frac{(-1)^{n-i-1}}{m'} (\omega_{i} + r\omega_{i-1})
      \end{equation}
      for all $2 \leq i \leq n-2$.

      Now we consider the determinant in \eqref{eq:val_of_most_important_entry}, we have
      \begin{align} \label{eq:ready_a_i_to_alpha_i}
         {
            \left| \begin{smallmatrix}
               0        & d_1 & d_2 & \ldots & d_{n-2} & d_{n-1} \\
               a_1      & r   &     &        &         & -\frac{x_{n-1}}{m'} \\
               a_2      & m'   &     &        &         & -x_{n-2} \\
               a_3      &     & 1   &        &         & -x_{n-3} \\
               \vdots   &     &     & \ddots &         & \vdots \\
               a_{n-1}  &     &     &        & 1       & -x_1 \\
            \end{smallmatrix} \right|
         }
         &= \frac{1}{m'} {
            \left| \begin{smallmatrix}
               0 & d_{1} & m'd_{2} & \ldots & m'd_{n-2} & m'd_{n-1} \\
               m'a_1 & r & & & & -x_{n-1} \\
               a_2 & 1 & & & & -x_{n-2} \\
               a_3 & & 1 & & & -x_{n-3} \\
               \vdots & & & \ddots & & \vdots \\
               a_{n-1} & & & & 1 & -x_1 \\
            \end{smallmatrix} \right|
         }.
      \end{align}
      By the relations in \eqref{eq:x_in_r}, \eqref{eq:def_of_a_i} and the definition in \eqref{def:x_j}, apply suitable row operations,  \eqref{eq:ready_a_i_to_alpha_i} is equal to
      \begin{equation} \label{eq:a_i_to_alpha_i}
         \frac{1}{m'} {
            \left| \begin{smallmatrix}
               0 & d_{1} & m'd_{2} & \ldots & m'd_{n-2} & m'd_{n-1} \\
               \alpha_0 & 0 & & & & -1 \\
               \alpha_1 & 1 & -r & & & -1 \\
               \alpha_2 & & 1 & & & -1 \\
               \vdots & & & \ddots & & \vdots \\
               \alpha_{n-2} & & & & 1 & -x_1 \\
            \end{smallmatrix} \right|
         }.
      \end{equation}
      Using \eqref{eq:1st_and_last_of_d_i_in_R} and \eqref{eq:d_i_in_R}, the determinant in \eqref{eq:a_i_to_alpha_i} becomes
      \begin{align} \label{eq:det_of_ready_to_simplify_to_R}
         &\quad {
            \left|
            \begin{smallmatrix}
               0 & (-1)^{n-2}\omega_1 & (-1)^{n-3}(\omega_2 + r\omega_1) & \ldots
                                         & -(\omega_{n-2} + r\omega_{n-3})
                                         & (\omega_{n-1} + r\omega_{n-2}) \\
               \alpha_0 & 0 & & & & -1 \\
               \alpha_1 & 1 & -r & & & -1 \\
               \alpha_2 & & 1 & & & -1 \\
               \vdots & & & \ddots & & \vdots \\
               \alpha_{n-2} & & & & 1 & -x_1 \\
            \end{smallmatrix}
            \right|
         } \nonumber \\
         &= {
            \left|
            \begin{smallmatrix}
               0 & (-1)^{n-2}\omega_1 & (-1)^{n-3}\omega_2 & \ldots & -\omega_{n-2} & \omega_{n-1} \\
               \alpha_0 & 0 & & & & -1 \\
               \alpha_1 & 1 & 0 & & & -1 \\
               \alpha_2 & & 1 & & & -1 \\
               \vdots & & & \ddots & & \vdots \\
               \alpha_{n-2} & & & & 1 & -1 \\
            \end{smallmatrix}
            \right|
         }. \text{ (by column operations) }
      \end{align}
      Expand the determinant in \eqref{eq:det_of_ready_to_simplify_to_R} along the first row, 
      and from \eqref{eq:ready_a_i_to_alpha_i}-\eqref{eq:det_of_ready_to_simplify_to_R}, 
      we get
      \begin{align*}
         &\quad b_{n-1,n-2}^{(n-2)} \\
         &= {
            \dfrac{(-1)^{n-1}}{m'c} (-1)^{n-1} \{
               (-1)^{n+1}(\alpha_1-\alpha_0)\omega_1 +
               (-1)^{n+2}(\alpha_2-\alpha_0)\omega_2 + \cdots + 
         } \\
         &\qquad \qquad \qquad {
               (-1)^{n+(n-2)}(\alpha_{n-2}-\alpha_0)\omega_{n-2} +
               \alpha_0\omega_{n-1}
            \}
         } \\
         &= {
            \dfrac{(-1)^{n-1}}{m'c} \{
               (\alpha_1-\alpha_0)\omega_1 -
               (\alpha_2-\alpha_0)\omega_2 + \cdots +
         } \\
         &\qquad \qquad \qquad {
               (-1)^{n-3}(\alpha_{n-2}-\alpha_0)\omega_{n-2} +
               (-1)^{n-2}(-\alpha_0)\omega_{n-1}
            \}
         } \\
         &= {
            \dfrac{1}{m'c}
            \left|
               \begin{array}{cccc:c}
                  & & & & \alpha_0-\alpha_1 \\
                  & & & & \alpha_0-\alpha_2 \\
                  & \mathcal{R} & & & \vdots \\
                  & & & & \alpha_0-\alpha_{n-2} \\
                  & & & & \alpha_0
               \end{array}
               \right|
         }, \text{ (by definition of $\omega_i$)}
      \end{align*}
      where $c = e_{n-3}^{2} \cdots e_{1}^{n-2}$.

      In the following, we shall compute $\mathcal{R}$ explicitly.

      Recall that
      $
         \mathcal{R} = {
            \left[ \begin{smallmatrix}
               m'\mathfrak{h}^{1}-r\mathfrak{h}^{2} \\ \mathfrak{h}^{2}-r\mathfrak{h}^{3} \\ \vdots \\ \mathfrak{h}^{n-2}-r\mathfrak{h}^{n-1} \\ \mathfrak{h}^{n-1}
            \end{smallmatrix} \right]
         }
      $ and
      $
         \mathcal{H} = {
               [\mathfrak{p}, \Delta\mathfrak{p}, \ldots, \Delta^{n-3}\mathfrak{p}]
         }
         = {
            \left[ \begin{smallmatrix} \mathfrak{h}^{1} \\ \mathfrak{h}^{2} \\ \vdots \\ \mathfrak{h}^{n-1} \end{smallmatrix} \right]
         }
      $.
      To compute $\mathcal{R}$ explicitly, we shall first compute $\Delta^{i}$ explicitly.

      The powers of the matrix $\Delta$ are given by
      \begin{equation} \label{eq:pow_of_Delta}
	      \Delta^{i} = { \left[
	         \mathfrak{b}_{i}, \mathfrak{e}_{i+2}, \mathfrak{e}_{i+3},
	         \ldots, \mathfrak{e}_{n-1}, \mathfrak{f}, \mathfrak{g},
	         \mathfrak{e}_{2}, \mathfrak{e}_{3}, \ldots, \mathfrak{e}_{i-1}
	      \right] },
      \end{equation}
      where $2 \leq i \leq n-2$ and
      \begin{displaymath}
	      \mathfrak{b}_{i} = {
	         \left[ \begin{smallmatrix}
	            r^{i} \\ r^{i-1}m' \\ r^{i-2}m' \\ \vdots \\ m' \\ 0 \\ \vdots \\ 0
	         \end{smallmatrix} \right]
	      },
	      \mathfrak{f} = {
	         \left[ \begin{smallmatrix}
	            -\frac{x_{n-1}}{m'} \\ -x_{n-2} \\ -x_{n-3} \\ \vdots \\ -x_{1}
	         \end{smallmatrix} \right]
	      }
         \textnormal{ and }
	      \mathfrak{g} = {
	         \left[ \begin{smallmatrix}
	            \frac{x_{n-1}}{m'} \\ rx_{n-3} \\ rx_{n-4} \\ \vdots \\ rx_{1} \\ r
	         \end{smallmatrix} \right]
	      }.
      \end{displaymath}
	   The verification of \eqref{eq:pow_of_Delta} is straightforward and can be done by induction and using the relations:
      \begin{displaymath}
         x_{i}x_{n-j} - rx_{i-1}x_{n-j-1} = x_{n+i-j}, 
      \end{displaymath}
      where $i \geq 2$ and $j \leq n-2$.

      From \eqref{eq:pow_of_Delta} and the relations
      \begin{displaymath}
         ra_{i} - a_{i-1} + \alpha_{i-2} = 0
         \quad \textnormal{for } 3 \leq i \leq n-1,
      \end{displaymath}
      the entries of $\mathcal{R}$ can be computed explicitly.
      For example, the entries in the first row of $\mathcal{R}$ are
      \begin{align*}
         r_{1j} &= {
            m'a_{1}r^{j-1} - a_{n-j+1}x_{n-1} + a_{n-j+2}x_{n-1}
         } \\
         &\quad {
            -ra_{1}r^{j-2}m' + ra_{n-j+1}x_{n-2} - r^{2}a_{n-j+2}x_{n-3} - ra_{n-j+3}
         } \\
         &= {
            -\alpha_{n-j} + \alpha_{n-j+1}
         },
      \end{align*}
      for $j \geq 3$. Thus, we conclude that
      \begin{align} \label{eq:det_val_of_most_important_entry}
         b_{n-1,n-2}^{(n-2)} &= {
            \dfrac{1}{m'c}
            \left| \begin{smallmatrix}
               \alpha_{0} & -\alpha_{n-2} & \alpha_{n-2}-\alpha_{n-3} & \alpha_{n-3}-\alpha_{n-4} & \cdots & \alpha_{0}-\alpha_{1} \\
               \alpha_{1} & \alpha_{0}-\alpha_{n-2} & -\alpha_{n-3} & \alpha_{n-2}-\alpha_{n-4} & \cdots & \alpha_{0}-\alpha_{2} \\
               \vdots & \vdots & \vdots & \vdots & & \vdots \\
               \alpha_{n-3} & \alpha_{n-4}-\alpha_{n-2} & \alpha_{n-5}-\alpha_{n-3} & \alpha_{n-6}-\alpha_{n-4} & \cdots & \alpha_{0}-\alpha_{n-2} \\
               \alpha_{n-2} & \alpha_{n-3}-\alpha_{n-2} & \alpha_{n-4}-\alpha_{n-3} & \alpha_{n-5}-\alpha_{n-4} & \cdots & \alpha_{0} \\
            \end{smallmatrix} \right|
         }.
      \end{align}
   \end{step}

   Let    
   $
      x = {
         \alpha_0 + \alpha_{1}\zeta_n + \cdots + \alpha_{n-2}\zeta_n^{n-2}
      } \in \mathbb{Z}[\zeta_n]
   $, the norm of $x$ is
   \begin{equation} \label{eq:norm_of_x}
      N_{\mathbb{Q}(\zeta_n)/\mathbb{Q}}(x) = {
            \left| \begin{smallmatrix}
               \alpha_{0} & -\alpha_{n-2} & \alpha_{n-2}-\alpha_{n-3} & \alpha_{n-3}-\alpha_{n-4} & \cdots & \alpha_{2}-\alpha_{1} \\
               \alpha_{1} & \alpha_{0}-\alpha_{n-2} & -\alpha_{n-3} & \alpha_{n-2}-\alpha_{n-4} & \cdots & \alpha_{3}-\alpha_{1} \\
               \vdots & \vdots & \vdots & \vdots & & \vdots \\
               \alpha_{n-3} & \alpha_{n-4}-\alpha_{n-2} & \alpha_{n-5}-\alpha_{n-3} & \alpha_{n-6}-\alpha_{n-4} & \cdots & -\alpha_{1} \\
               \alpha_{n-2} & \alpha_{n-3}-\alpha_{n-2} & \alpha_{n-4}-\alpha_{n-3} & \alpha_{n-5}-\alpha_{n-4} & \cdots & \alpha_{0}-\alpha_{1} \\
            \end{smallmatrix} \right|
      }.
   \end{equation}
   By adding all the first $n-2$ columns to last column of \eqref{eq:norm_of_x}, we can conclude that the determinant in \eqref{eq:det_val_of_most_important_entry} is actually the norm of $x$; that is, 
   \begin{displaymath}
      b_{n-1,n-2}^{(n-2)} = {
         \frac{N_{\mathbb{Q}(\zeta_n)/\mathbb{Q}}(x)}{m'c} 
      }.
   \end{displaymath}

Now, from assumption that $N_{\mathbb{Q}(\zeta_n)/\mathbb{Q}}(x) = m'$, we get
\begin{displaymath}
   b_{n-1,n-2}^{(n-2)} = {
      \frac{N_{\mathbb{Q}(\zeta_n)/\mathbb{Q}}(x)}{m'e_{n-3}^{2} \cdots e_{1}^{n-2}}
   }
   = {
      \frac{1}{e_{n-3}^{2} \cdots e_{1}^{n-2}}
   }.
\end{displaymath}
Since $\mathcal{B}_{j}$ are integral matrices and $e_{i} \in \mathbb{N}$, we conclude that $e_1 = \cdots = e_{n-3} = b_{n-1,n-2}^{(n-2)} = 1$, which completes the proof of the Main Theorem.

\section{Corollaries and Examples} \label{sec:cors_and_examples}
In fact, the most important case of Main Theorem is the following theorem.
\begin{thm} \label{thm:most_important_case}
   Let $p, q$ be odd primes, $K$ be a field such that neither $p$ nor $q$ divides the characteristic of $K$ and both $\zeta_p, \zeta_q$ lie in $K$.
   Let
   \begin{displaymath}
      G = C_p \rtimes_r C_q = \langle \sigma_1, \sigma_2 : \sigma_1^p = \sigma_2^q = 1, \sigma_2^{-1}\sigma_1\sigma_2 = \sigma_1^r \rangle,
   \end{displaymath}
   where $r^q \equiv 1 \pmod{p}$.
   If there exists $x \in \mathbb{Z}[\zeta_q]$ such that the norm $N_{\mathbb{Q}(\zeta_q)/\mathbb{Q}}(x) = p$, then $K(G)$ is rational over $K$.
\end{thm}
\begin{proof}
   Suppose $x = \alpha_0 + \alpha_1 \zeta_q + \cdots + \alpha_{q-2} \zeta_{q}^{q-2} \in \mathbb{Z}[\zeta_q]$ satisfies $N_{\mathbb{Q}(\zeta_q)/\mathbb{Q}}(x) = p$. Let $d = \gcd \{ \alpha_0, \ldots, \alpha_{q-2} \}$, we have $d^{q-1}|N_{\mathbb{Q}(\zeta_q)/\mathbb{Q}}(x) = p$. It follows that $d = 1$.

   Now we show that if $r \in (\mathbb{Z}/p\mathbb{Z})^{\times}$ is of order $q$ and there exists
   $
      x = \alpha_0 + \alpha_1 \zeta_q + \cdots + \alpha_{q-2} \zeta_{q}^{q-2} \in \mathbb{Z}[\zeta_q]
   $
   satisfying $N_{\mathbb{Q}(\zeta_n)/\mathbb{Q}}(x) = p$, then there exists $k \in \mathbb{N}$ such that
   \begin{displaymath}
      \alpha_0 + \alpha_1 r^{k} + \cdots + \alpha_{q-2} r^{k(q-2)} \equiv 0 \pmod{p}.
   \end{displaymath}

   In fact, let $\tilde{\phi}:\mathbb{Z}[X] \rightarrow \mathbb{Z}/p\mathbb{Z}$ be the ring homomorphism such that $\tilde{\phi}(X) = r$.
   Since $r^{q-1} + \cdots + r + 1 \equiv 0 \pmod{p}$, we have $\langle X^{q-1} + \cdots + X + 1 \rangle \subseteq \ker \tilde{\phi}$.
   Hence there exists a well-defined ring homomorphism $\phi:\mathbb{Z}[\zeta_q] \rightarrow \mathbb{Z}/p\mathbb{Z}$, which maps $\zeta_q$ to $r$.

   In $\mathbb{Q}(\zeta_q)$, we have
   \begin{align*}
      \mathbb{Z}/p\mathbb{Z} \ni 0 &= \phi(N_{\mathbb{Q}(\zeta_q)/\mathbb{Q}}(x)) \\
      &= {
         \phi(\alpha_0 + \alpha_1 \zeta_q + \cdots + \alpha_{q-2} \zeta_{q}^{q-2}) \cdot
         \phi(\alpha_0 + \alpha_1 \zeta_q^{2} + \cdots + \alpha_{q-2} \zeta_{q}^{2(q-2)})
         \cdots
      } \\
      &{
         \qquad
         \phi(\alpha_0 + \alpha_1 \zeta_q^{q-1} + \cdots + \alpha_{q-2} \zeta_{q}^{(q-2)(q-1)})
      } \\
      &= {
         (\alpha_0 + \alpha_1 r + \cdots + \alpha_{q-2} r^{q-2}) \cdot
         (\alpha_0 + \alpha_1 r^{2} + \cdots + \alpha_{q-2} r^{2(q-2)})
         \cdots
      } \\
      &{
         \qquad
         (\alpha_0 + \alpha_1 r^{q-1} + \cdots + \alpha_{q-2} r^{(q-2)(q-1)})
      }.
   \end{align*}
   Hence there exists $k \in \mathbb{N}$ such that
   \begin{displaymath}
      \alpha_0 + \alpha_1 r^{k} + \cdots + \alpha_{q-2} r^{k(q-2)} \equiv 0 \pmod{p}.
   \end{displaymath}
   
   Note that $G = C_{p} \rtimes_{r} C_{q}$ is unique up to isomorphism, the choice of $r$ can be arbitrary.
   Hence we may replace $r$ by $r^{k}$, which completes the proof.
\end{proof}

\begin{cor}[C.f. \cite{Kang2009-rpfsmg}] \label{cor:kangs_result}
   Let $p, q$ be odd primes such that $\mathbb{Z}[\zeta_q]$ is a unique factorization domain.
   Let $K$ be a field such that neither $p$ nor $q$ divides the characteristic of $K$ and both $\zeta_p, \zeta_q$ lie in $K$.
   Let
   \begin{displaymath}
      G = C_p \rtimes_r C_q = \langle \sigma_1, \sigma_2 : \sigma_1^p = \sigma_2^q = 1, \sigma_2^{-1}\sigma_1\sigma_2 = \sigma_1^r \rangle,
   \end{displaymath}
   where $r^q \equiv 1 \pmod{p}$.
   Then $K(G)$ is rational over $K$.

   In fact, $\mathbb{Z}[\zeta_q]$ is unique factorization domain if and only if $q < 23$.
\end{cor}
\begin{proof}
   \setcounter{step}{0}
   By Theorem \ref{thm:most_important_case}, it suffices to show that there exists $x \in \mathbb{Z}[\zeta_q]$ such that $N_{\mathbb{Q}(\zeta_q)/\mathbb{Q}}(x) = p$.

%

	Note that since $r^{q} \equiv 1 \pmod{p}$, $\mathbb{Z}/p\mathbb{Z}$ contains a primitive $q$th-root of unity, and therefore the polynomial $X^{q-1} + \cdots + X + 1$ splits into linear factors in $\mathbb{Z}/p\mathbb{Z}$.
   Since $\mathbb{Z}[\zeta_q]$ is a unique factorization domain, it's a principal ideal domain.
   Write $p\mathbb{Z}[\zeta_q] = \mathfrak{p}_{1} \cdots \mathfrak{p}_{q-1}$, where
   $\mathfrak{p}_i = \langle \pi_i \rangle$ for $i = 1, \ldots, q-1$.
   Then each pair of $\pi_{i}, \pi_{j}$ are conjugate and
   \begin{align*}
      p\mathbb{Z} &= {
         \langle \pi_1 \cdots \pi_{q-1} \rangle
      }
      = {
         \langle \prod_{\sigma} \sigma(\pi_{1}) \rangle
      }
      = {
         \langle N_{\mathbb{Q}(\zeta_q)/\mathbb{Q}}(\pi_1) \rangle
      }.
   \end{align*}
   By positivity of norms, we get $p = N(\pi_1)$.
\end{proof}

In the following, we give a class of semidirect product groups, which are not semidirect product of two simple cyclic groups, but do satisfy the conditions of the Main Theorem.

\begin{cor} \label{cor:class_of_semidirect_prod_of_nonsimple_grps}
   Let $q$ be an odd prime, $m = \alpha q^{k}$, where $k \geq 2$ and $q \nmid \alpha$.
   Let
   \begin{displaymath}
      G = C_m \rtimes_r C_q = \langle \sigma_1, \sigma_2 : \sigma_1^m = \sigma_2^q = 1, \sigma_2^{-1}\sigma_1\sigma_2 = \sigma_1^r \rangle,
   \end{displaymath}
   where $r = \alpha q^{k-1} + 1$.
   Then $\mathbb{C}(G)$ is rational over $\mathbb{C}$.
\end{cor}
\begin{proof}
   It's clear that $r^{q} = (\alpha q^{k-1} + 1)^{q} \equiv 1 \pmod m$ and $m' = \frac{m}{\gcd(m, r-1)} = q$.

   Put $x = 1-\zeta_{q}$. 
   Note that $X^{q-1} + \cdots + X + 1 = (X-\zeta_{q})(X-\zeta_{q}^{2})\cdots(X-\zeta_{q}^{q-1})$.
   Substitute $X$ by $1$, we get
   \begin{align*}
      N_{\mathbb{Q}(\zeta_{q})/\mathbb{Q}}(1-\zeta_{q})
      = {
         (1-\zeta_{q})(1-\zeta_{q}^{2})\cdots(1-\zeta_{q}^{q-1})
      }
      = q.
   \end{align*}
   Hence by Main Theorem, $\mathbb{C}(G)$ is rational over $\mathbb{C}$.
\end{proof}

We use a computer to find valid pairs $p, q$ and elements $x$ for the Main Theorem.
Below, we list some valid examples, which are not covered by Corollary \ref{cor:kangs_result},
but are rational by the Main Theorem.
\begin{example}
The following triples conform to the conditions of Theorem \ref{thm:most_important_case}: \\
$
(q, p, x) = 
$ \\
$
   (29, 5801, 1 + \zeta_{29} + \zeta_{29}^{4}),
   (29, 4931, 1 - \zeta_{29}^2 + \zeta_{29}^5), 
   (29, 7193, 1 + \zeta_{29}^2 + \zeta_{29}^5), 
$ \\
$
   (29, 9803, -1 + \zeta_{29} + \zeta_{29}^{4}), 
   (29, 12413, -1 + \zeta_{29}^2 + \zeta_{29}^5), 
   (29, 18097, 1 + \zeta_{29} + \zeta_{29}^4), 
$ \\
$
   (29, 18503, 1 - \zeta_{29} + \zeta_{29}^5), 
   (29, 21577, 1 + \zeta_{29}^2 + \zeta_{29}^3), 
   (31, 5953, -1 - \zeta_{31} + \zeta_{31}^{3}), 
$ \\
$
   (31, 6263, 1 - \zeta_{31} + \zeta_{31}^{3}), 
   (31, 11657, 1 + \zeta_{31} + \zeta_{31}^4), 
   (31, 16741, -1 - \zeta_{31} + \zeta_{31}^4), 
$ \\
$
   (31, 20089, -1 + \zeta_{31} + \zeta_{31}^6), 
   (37, 32783, 1 - \zeta_{37} + \zeta_{37}^{3}), 
   (37, 68821, -1 + \zeta_{37}^2 + \zeta_{37}^5), 
$ \\
$
   (37, 108929, 1 + \zeta_{37}^2 + \zeta_{37}^5), 
   (37, 132313, -1 + \zeta_{37} + \zeta_{37}^4), 
   (37, 172717, -1 - \zeta_{37} + \zeta_{37}^4), 
$ \\
$
   (37, 262553, 1 - \zeta_{37}^3 + \zeta_{37}^4), 
   (41, 101107, -1 - \zeta_{41} + \zeta_{41}^{3}), 
   (41, 337759, 1 + \zeta_{41} + \zeta_{41}^4), 
$ \\
$
   (41, 340793, -1 + \zeta_{41}^2 + \zeta_{41}^5), 
   (41, 348911, 1 - \zeta_{41}^2 + \zeta_{41}^5), 
   (41, 432059, 1 + \zeta_{41}^2 + \zeta_{41}^5).
$
\end{example}

\section*{Acknowledgement}
We are grateful to the referee for not only careful reading our paper, 
but also providing a simplified proof to Corollary \ref{cor:kangs_result}
and informing us the paper written by E. Beneish and N.Ramsey.
These improve the quality of our paper.

   \clearpage
   \bibliographystyle{elsart-num-sort}
   \bibliography{export}
\end{document}